\documentclass[12pt]{amsart}

\usepackage{amsmath}
\usepackage{amssymb}
\usepackage{amsthm}
\usepackage{amscd}
\usepackage{xypic}
\swapnumbers \textwidth=16.00cm \textheight=22cm \topmargin=0.00cm
\headsep=1cm \numberwithin{equation}{section}
\newtheorem{theorem}{Theorem}[section]
\newtheorem{lemma}[theorem]{Lemma}
\newtheorem{proposition}[theorem]{Proposition}
\newtheorem{corollary}[theorem]{Corollary}

\theoremstyle{definition}

\newtheorem{definition and remark}[theorem]{Definition and Remark}
\newtheorem{remark}[theorem]{Remark}
\newtheorem{remark and definition}[theorem]{Remark and Definition}
\newtheorem{remark and notation}[theorem]{Remark and Notation}
\newtheorem{notation and remark}[theorem]{Notation and Remark}
\newtheorem{notation and convention}[theorem]{Notation and Convention}

\newtheorem{notation and remarks}[theorem]{Notation and Remarks}
\newtheorem{notation and reminder}[theorem]{Notation and Reminder}

\newtheorem{construction and examples}[theorem]{Construction and Examples}

\newtheorem{problem and remark}[theorem]{Problem and Remark}

\newcommand\reg{\operatorname{reg}}

\begin{document}

\title[PROJECTIVE VARIETIES OF MAXIMAL SECTIONAL REGULARITY]
         {PROJECTIVE VARIETIES OF MAXIMAL SECTIONAL REGULARITY}

\author{Markus BRODMANN, Wanseok LEE, Euisung PARK, Peter SCHENZEL}

\address{Universit\"at Z\"urich, Institut f\"ur Mathematik, Winterthurerstrasse 190, CH -- Z\"urich, Switzerland}
\email{brodmann@math.unizh.ch}

\address{Department of Applied Mathematics, Pukyong National University, Busan 608-737, Korea}
\email{wslee@pknu.ac.kr}

\address{Korea University, Department of Mathematics, Anam-dong, Seongbuk-gu, Seoul 136-701, Republic of Korea}
\email{euisungpark@korea.ac.kr}

\address{Martin-Luther-Universit\"at Halle-Wittenberg,
Institut f\"ur Informatik, Von-Secken\-dorff-Platz 1, D -- 06120
Halle (Saale), Germany} \email{schenzel@informatik.uni-halle.de}

\date{Seoul, 06. January 2015}



\begin{abstract}
We study projective varieties $X \subset \mathbb{P}^r$ of dimension
$n \geq 2$, of codimension $c \geq 3$ and of degree $d \geq c + 3$
that are of maximal sectional regularity, i.e. varieties for which
the Castelnuovo-Mumford regularity $\reg (\mathcal{C})$ of a general
linear curve section is equal to $d -c+1$, the maximal possible
value (see \cite{GruLPe}). As one of the main results we classify
all varieties of maximal sectional regularity. If $X$ is a variety
of maximal sectional regularity, then either (a) it is a divisor on
a rational normal $(n+1)$-fold scroll $Y \subset \mathbb{P}^{n+3}$
or else (b) there is an $n$-dimensional linear subspace $\mathbb{F}
\subset \mathbb{P}^r$ such that $X \cap \mathbb{F} \subset
\mathbb{F}$ is a hypersurface of degree $d-c+1$. Moreover, suppose
that $n = 2$ or the characteristic of the ground field is zero. Then
in case (b) we obtain a precise description of $X$ as a birational
linear projection of a rational normal $n$-fold scroll.
\end{abstract}

\maketitle \thispagestyle{empty}

\section{Introduction}

\noindent Let $X \subset \mathbb{P}^r$ be a nondegenerate
irreducible projective variety of dimension $n$, codimension $c>1$
and degree $d$ over an algebraically closed field $\Bbbk$. D.
Mumford\cite{M} has defined $X$ to be $m$-\textit{regular} if its
ideal sheaf $\mathcal{I}_X$ satisfies the following vanishing
condition
\begin{equation*}
H^i(\mathbb{P}^r, \mathcal{I}_X(m-i)) = 0 \mbox{ for all } i \geq 1.
\end{equation*}
The $m$-regularity condition implies the $(m+1)$-regularity
condition, so that one defines the \textit{Castelnuovo-Mumford
regularity} $\reg(X)$ of $X$ as the least integer $m$ such that $X$
is $m$-regular. It is well known that if $X$ is $m$-regular then its
homogeneous ideal is generated by forms of degree $\leq m$. This
algebraic implication of $m$-regularity has an elementary geometric
consequence that any $(m+1)$-secant line to $X$ should be contained
in $X$. We say that a linear space $L \subset \mathbb{P}^r$ is
$k$-secant to $X$ if
\begin{equation*}
\rm{length}(X \cap L) := \rm{dim}_{\Bbbk}
(\mathcal{O}_{\mathbb{P}^r} / \mathcal{I}_X + \mathcal{I}_L) \geq k.
\end{equation*}

A well known conjecture due to Eisenbud and Goto (see \cite{EG})
says that
\begin{equation}
\reg(X) \leq d-c+1.
\end{equation}
Obviously this conjecture implies the following conjecture
\begin{equation}
X ~ \mbox{has no proper $k$-secant line if $k>d-c+1$.}
\end{equation}
So far the conjecture (1.1) has been proved only for irreducible but
not necessarily smooth curves by
Gruson-Lazarsfeld-Peskine\cite{GruLPe} and for smooth complex
surfaces by H. Pinkham\cite{Pi} and R. Lazarsfeld\cite{L}. Moreover,
in \cite{GruLPe} the curves in $\mathbb{P}^r$ whose regularity takes
the maximally possible value $d-r+2$ are completely classified: they
are either of degree $\leq r+1$ or else smooth rational curves
having a $(d-r+2)$-secant line. The statement (1.2) is known to be
true when $X$ is locally Cohen-Macaulay (see Theorem 1 in \cite{N}).
But it is still unknown for arbitrary varieties.

The main subject of the present paper is to study the geometry of
proper $(d-c+1)$-secant lines to a projective variety. To this aim,
we investigate the \textit{extremal secant locus} $\Sigma (X)$ of
$X$, that is, the closure of the set of all proper $(d-c+1)$-secant
lines to $X$ in the Grassmannian $\mathbb{G}(1,\mathbb{P}^r)$. Of
course, if the extremal secant locus of $X$ is nonempty then its
regularity is at least $d-c+1$ and so such a variety will play an
important role in the natural problem of classifying all extremal
varieties with respect to the above regularity conjecture. For $d
\geq c+3$, Gruson-Lazarsfeld-Peskine's result in \cite{GruLPe}
provides a complete classification of curves having a
$(d-c+1)$-secant line. They should be smooth and rational. M. A.
Bertin\cite{Be} generalizes this result to higher dimensional smooth
varieties. She proves the conjecture (1.1) for smooth rational
scrolls -- which is reproved in \cite{KwP} -- and shows that if $X$
is a smooth variety having a $(d-c+1)$-secant line then it should be
the linear regular projection of a smooth rational normal scroll.
Later, A. Noma\cite{N} obtains a very nice description of those
smooth rational scrolls.

In Theorem \ref{thm:maximaldimension} we show that if $c \geq 3$ and
$d \geq c+3$, then the dimension of $\Sigma (X)$ is at most $2n-2$
and the equality is attained if and only if a general linear curve
section of $X$ has the maximal Castelnuovo-Mumford regularity
$d-c+1$. We will say that $X$ is a \textit{variety of maximal
sectional regularity} if its general linear curve section is of
maximal regularity (cf. \cite{BS5}).

To complete the result starting with Theorem
\ref{thm:maximaldimension}, it is natural to ask for a
classification of all varieties of maximal sectional regularity.
This is the contents of Theorem \ref{thm:classificationsurface} and
Theorem \ref{thm:classificationhigherdimensional}. More precisely,
for $c \geq 3$ and $d \geq c+3$ we obtain a classification of
surfaces of maximal sectional regularity in Theorem
\ref{thm:classificationsurface} and a classification of higher
dimensional varieties of maximal sectional regularity in Theorem
\ref{thm:classificationhigherdimensional}. It turns out that $X
\subset \mathbb{P}^r$ is variety of maximal sectional regularity if
and only if it is one of the followings:
\begin{itemize}
\item[(a)] $c=3$ and $X$ is a divisor of the $(n+1)$-fold scroll $Y = S(\underbrace{0,\ldots,0}_{(n-2)-\rm{times}},1,1,1) \subset \mathbb{P}^{n+3}$ such that $X$ is linearly equivalent to $H + (d-3)F$, where $H$ is the hyperplane divisor of $Y$ and $F \subset Y$ is a linear subspace of dimension $n$;
\item[(b)] There exists an $n$-dimensional linear subspace $\mathbb{F} \subset \mathbb{P}^r$ such that $X \cap \mathbb{F}$ in $\mathbb{F}$ is a hypersurface of degree $d-c+1$.
\end{itemize}
In particular, there exist varieties $X \subset \mathbb{P}^r$ of
maximal sectional regularity of dimension $n$, of codimension $c$
and of degree $d$ for any given $(n,c,d)$ with $n \geq 2$, $c \geq
3$ and $d \geq c+3$. Furthermore, assume that $\rm{char}(\Bbbk)=0$
or $n=2$. Then in case (b), we obtain a very precise description of
$X$ as a birational linear projection of a rational normal $n$-fold
scroll. See Theorem \ref{thm:classificationsurface} and Theorem
\ref{thm:classificationhigherdimensional}.

\begin{remark}
Let $X \subset \mathbb{P}^r$ be a nondegenerate irreducible
projective variety of dimension $n \geq 2$, codimension $c \geq 2$
and degree $d$. Thus $d \geq c+1$.

\noindent $(1)$ Our subject of the present paper is quite well
understood if $d \leq c+2$. More precisely, varieties of minimal
degree (i.e. $d=c+1$) are characterized by the $2$-regularity (cf.
\cite{EG}). Of course, these varieties have many proper secant
lines. If $d=c+2$, then $\rm{reg}(X)=3$ but $X$ may be cut out by
quadrics and so it may have no tri-secant lines (cf. \cite{HSV},
\cite{P2}).

\noindent $(2)$ One can naturally ask whether $X$ satisfies the
regularity bound in (1.1) when it has a nonempty extremal secant
locus and hence $\rm{reg}(X) \geq d-c+1$. By M. A. Bertin's work in
\cite{Be}, the answer for this question is ``YES" when $X$ is
smooth. But it is unknown if $X$ is a singular variety. In this
direction, the authors in \cite{BLPS1} study various cohomological
and homological properties of $X$ when it is a surface of maximal
sectional regularity. In particular, it is shown that such a surface
achieves the regularity bound in (1.1).
\end{remark}

\section{Curves of maximal regularity}

\noindent In this section we prove some results on curves of maximal
regularity, which will be useful for our later investigations. We
first fix a few notations which we shall keep for the rest of this
paper.

\begin{notation and remarks}\label{nar:curveofmaxreg}
Let $\mathcal{C} \subset \mathbb{P}^r$ be a nondegenerate
projective integral curve of degree $d$.
\begin{itemize}
\item[\rm{(A)}] In their fundamental paper (cf. \cite{GruLPe}) Gruson,
Lazarsfeld and Peskine have shown that
\begin{equation*}
\reg(\mathcal{C}) \leq d-r+2
\end{equation*}
and the equality is attained if and only if either $d \leq r+1$ or
else $d \geq r+2$ and $\mathcal{C}$ is a smooth rational curve
having a $(d-r+2)$-secant line. We say that $\mathcal{C}$ is of
\textit{maximal regularity} if $\reg(\mathcal{C}) = d-r+2$.

\item[\rm{(B)}] If $d \geq r+2$ and $\mathcal{C}$ is a curve of maximal
regularity, then $\mathcal{C}$ is the regular projection of a
rational normal curve $\widetilde{\mathcal{C}} \subset \mathbb{P}^d$
of degree $d$ and hence $\mathcal{C}$ is not linearly normal.

\item[\rm{(C)}] Let $r \geq 4$ and $d \geq r+2$. According to \cite[Remark 3.1(C)]{BS2}, if $\mathcal{C}$ is
of maximal regularity then the $(d-r+2)$-secant line of statement
(A) is uniquely determined. We will denote this line by
$\mathbb{L}_{\mathcal{C}}$. Throughout this section we will see that
this line induces additional geometric properties of $\mathcal{C}$.
\end{itemize}
\end{notation and remarks}

\begin{notation and remarks}\label{nar:descriptionscrolls}
We recall a standard description of rational normal scrolls (cf.
\cite{Sch}). For the vector bundle
\begin{equation*}
\mathcal{E} = \mathcal{O}_{\mathbb{P}^1} (a_1) \oplus \cdots \oplus
\mathcal{O}_{\mathbb{P}^1} (a_n)
\end{equation*}
on $\mathbb{P}^1$ where $0 \leq a_1 \leq \cdots \leq a_n$ and $a_n
> 0$, the tautological line bundle $\mathcal{O}_{\mathbb{P} (\mathcal{E})}
(1)$ of $\mathbb{P} ( \mathcal{E})$ is globally generated and we
write $S(a_1,\cdots,a_n)$ for the image of the map defined by
$\mathcal{O}_{\mathbb{P} (\mathcal{E})} (1)$.

\begin{itemize}
\item[\rm{(A)}] It is well-known that $S(a_1,\cdots,a_n)$ is a
normal variety and has only rational singularities. Also the
homogeneous ideal of $S(a_1,\cdots,a_n)$ is generated by quadrics.
In particular, any tri-secant line to $S(a_1,\cdots,a_n)$ is
contained in $S(a_1,\cdots,a_n)$.

\item[\rm{(B)}] Suppose that $n \geq 2$. Then the
divisor class group of $\mathbb{P} ( \mathcal{E})$ is freely
generated by $\widetilde{H} \in |\mathcal{O}_{\mathbb{P}
(\mathcal{E})} (1)|$ and a linear subspace $\widetilde{F}$ of
dimension $n-1$. Moreover, if $a_{n-1} >0$ then the morphism
$\varphi : \mathbb{P} ( \mathcal{E}) \rightarrow S(a_1,\cdots,a_n)$
induces an isomorphism between the divisor class groups. Thus the
divisor class group of $S(a_1,\cdots,a_n)$ is freely generated by
the hyperplane divisor $H$ and a linear subspace $F$ of dimension
$n-1$. We refer the reader to \cite{Fe}.

\item[\rm{(C)}] One can compute explicitly
the dimension of $H^i (\mathbb{P} (\mathcal{E}),
\mathcal{O}_{\mathbb{P} (\mathcal{E})}
(a\widetilde{H}+b\widetilde{F}))$ by using the projective bundle map
$j :\mathbb{P} (\mathcal{E}) \rightarrow \mathbb{P}^1$. For example,
see \cite[Section 2]{Miy}.
\end{itemize}
\end{notation and remarks}

\begin{proposition}\label{prop:2.3}
Let $r \geq 4$ and $d \geq r+2$. Let $\mathcal{C} \subset
\mathbb{P}^r$ be a curve of degree $d$ which is of maximal
regularity and let $\mathbb{L}_{\mathcal{C}}$ be as in Notation and
Remarks \ref{nar:curveofmaxreg}(B). Then

\begin{itemize}
\item[\rm{(a)}] ${\rm Join}(\mathbb{L}_{\mathcal{C}},\mathcal{C})$
is projectively equivalent to the threefold scroll $S(0,0,r-2)$.
\item[\rm{(b)}] Suppose that $\mathcal{C}$ is contained in a
rational normal threefold scroll $Y := S(0,0,r-2)$. Then $Y = {\rm
Join}(\mathbb{L}_{\mathcal{C}},\mathcal{C})$ and the vertex
$\mathbb{L} = S(0,0) \subset S(0,0,r-2)$ of $Y$ is equal to
$\mathbb{L}_{\mathcal{C}}$. In particular, ${\rm
Join}(\mathbb{L}_{\mathcal{C}},\mathcal{C})$ is the only rational
normal threefold scroll which is projectively equivalent to
$S(0,0,r-2)$ and which contains $\mathcal{C}$.
\end{itemize}
\end{proposition}

\begin{proof}
(a):  Choose a subspace $\mathbb{P}^{r-2} \subset \mathbb{P}^r$
which does not meet $\mathbb{L}_{\mathcal{C}}$ and consider the
linear projection map
\begin{equation*}
\pi_{\mathbb{L}_{\mathcal{C}}} : \mathbb{P}^r  \setminus
\mathbb{L}_{\mathcal{C}} \rightarrow \mathbb{P}^{r-2}
\end{equation*}
of $\mathcal{C}$ from $\mathbb{L}_{\mathcal{C}}$. Then $\mathcal{C}'
:= \overline{\pi_{\mathbb{L}_{\mathcal{C}}} ( \mathcal{C} \setminus
\mathbb{L}_{\mathcal{C}})}$ is a nondegenerate rational normal curve
in $\mathbb{P}^{r-2}$ and ${\rm
Join}(\mathbb{L}_{\mathcal{C}},\mathcal{C})$ is the cone over
$\mathcal{C}'$ with vertex $\mathbb{L}_{\mathcal{C}}$.
\smallskip

(b): We assume that $\mathbb{L} \neq \mathbb{L}_{\mathcal{C}}$ and
aim for a contradiction. As $\rm{length}(\mathcal{C} \cap
\mathbb{L}_\mathcal{C}) > 2$ we have $\mathbb{L}_\mathcal{C} \subset
Y$ and hence $\langle \mathbb{L}, \mathbb{L}_\mathcal{C} \rangle
\subset Y$, so that $\mathbb{L}$ and $\mathbb{L}_\mathcal{C}$ are
coplanar. Now, consider the linear projection map $\pi_{\mathbb{L}}:
\mathbb{P}^r \setminus \mathbb{L} \rightarrow \mathbb{P}^{r-2}$. The
restriction map $\pi_{\mathbb{L}} \upharpoonright: Y \setminus
\mathbb{L} \rightarrow S(r-2)$ induces a dominant morphism
$\mathcal{C} \setminus (\mathcal{C} \cap \mathbb{L}) \rightarrow
S(r-2)$. As $\mathcal{C}$ is smooth, this morphism may be extended
to a surjective morphism $\phi: \mathcal{C} \twoheadrightarrow
S(r-2)$. This implies that
\begin{equation*}
\deg(\phi) = \frac{d-\rm{length}(\mathcal{C} \cap
\mathbb{L})}{\deg_{\mathbb{P}^{r-2}}(S(r-2))} \leq \frac{d}{r-2}.
\end{equation*}
As $\mathbb{L}$ and $\mathbb{L}_\mathcal{C}$ are coplanar,
$\phi(\mathcal{C}\cap \mathbb{L}_{\mathcal{C}}) =
\pi_{\mathbb{L}}(\mathbb{L}_{\mathcal{C}}\setminus \mathbb{L})$ is a
point, say $z \in S(r-2)$. As $S(r-2)$ is smooth, this implies that
\begin{equation*}
\deg(\phi) = \rm{length} (\phi^{-1}(z)) \geq \rm{length}(\mathcal{C}
\cap \mathbb{L}_\mathcal{C}) = d-r+2.
\end{equation*}
The two previous inequalities imply that $\frac{d}{r-2} \geq d-r+2$,
which is impossible since $d \geq r+2$. This contradiction shows
that $\mathbb{L} = \mathbb{L}_{\mathcal{C}}$ and hence proves our
claim.
\end{proof}

Let $\mathcal{C} \subset \mathbb{P}^r$ be as in the above
Proposition \ref{prop:2.3}. In the next Proposition \ref{prop:2.4},
we study the case where our curve $\mathcal{C}$ lies on a smooth
rational normal surface. Note that the threefold scroll ${\rm
Join}(\mathbb{L}_{\mathcal{C}}, \mathcal{C})=S(0,0,r-2)$ contains
many smooth rational normal surface scrolls projectively equivalent
to $S(1,r-2)$. For example, any isomorphism from
$\mathbb{L}_{\mathcal{C}}=S(0,0)$ to $\mathcal{C}'$ in the above
proof defines a rational normal surface scroll of type $S(1,r-2)$
which is contained in ${\rm Join}(\mathbb{L}_{\mathcal{C}},
\mathcal{C})$. Also it may happen that $\mathcal{C}$ is contained in
such a surface scroll.

\begin{proposition}\label{prop:2.4}
Let $\mathcal{C} \subset \mathbb{P}^r$ be as in Proposition
\ref{prop:2.3}. If $\mathcal{C}$ is contained in a smooth rational
normal surface scroll $S = S(\alpha,r-\alpha-1)$ for some $1 \leq
\alpha \leq \frac{r-1}{2}$, then
\begin{itemize}
\item[(a)] $\alpha=1$,
\item[(b)] $\mathcal{C}$ is linearly equivalent to the divisor
$H + (d-r+1)F$ where $H$ and $F$ respectively are a general
hyperplane section and a ruling line of $S$,
\item[(c)] $\mathbb{L}_{\mathcal{C}}$ is equal to the unique line section $S(1)
\subset S(1,r-2)$ of $S$, and
\item[(d)] $S$ is contained in ${\rm
Join}(\mathbb{L}_{\mathcal{C}}, \mathcal{C})$.
\end{itemize}
\end{proposition}

\begin{proof}
Let $\mathcal{C}$ be linearly equivalent to $aH + bF$. Then $a \geq
1$ since $\mathcal{C}$ is irreducible and not a line. As the surface
$S$ is arithmetically Cohen-Macaulay, we have $H^i(\mathbb{P}^r,
\mathcal{I}_S(1)) = 0$ for $i=1,2$ and so, the short exact sequence
$$0 \longrightarrow \mathcal{I}_S \longrightarrow \mathcal{I}_{\mathcal{C}} \longrightarrow \mathcal{O}_S(-\mathcal{C}) \longrightarrow 0$$
implies an isomorphism $H^1(\mathbb{P}^r, \mathcal{I}_{\mathcal{C}}
(1)) \cong H^1 (S,\mathcal{O}_S ((1-a)H - bF))$. If $a >1$, we have
$H^1\big(\mathbb{P}^r, \mathcal{O}_S ((1-a)H - bF )\big) = 0$, and
we get the contradiction that $\mathcal{C} \subset \mathbb{P}^r$ is
linearly normal (cf. Notation and Remarks
\ref{nar:curveofmaxreg}(C)). Therefore $a=1$. As $d =
\deg(\mathcal{C}) = \deg(S) + b = r-1+b$, we obtain $b = d-r+1$ and
$\mathcal{C}$ is linearly equivalent to $H + (d-r+1)F$.

Now, we will see that $\alpha =1$. Note that $S$ is cut out by
quadrics and so $\mathbb{L}_{\mathcal{C}}$ should be contained in
$S$ (cf. Notation and Remarks \ref{nar:descriptionscrolls}(A)). If
$\alpha \geq 2$ then the only lines contained in $S$ are the ruling
lines. Obviously, no ruling line of $S$ can be a multi-secant line
to $\mathcal{C}$. Therefore $\alpha=1$.

The unique line section $\mathbb{L} = S(1)$ of $S = S(1,r-2)$
satisfies the condition
\begin{equation*}
\rm{length}(\mathcal{C} \cap \mathbb{L}) = \mathcal{C} \cdot
\mathbb{L} = (H +(d-r+1)F)\cdot (H-(r-2)F) = d-r+2,
\end{equation*}
and hence $\mathbb{L}$ is indeed the unique $(d-r+2)$-secant line to
$\mathcal{C}$. Now, it is clear that $S$ is contained in the
threefold scroll ${\rm Join}(S(1), \mathcal{C}) = {\rm
Join}(\mathbb{L}_{\mathcal{C}}, \mathcal{C})$.
\end{proof}

\section{The extremal secant locus of a projective variety}

\noindent In this section, we study the geometry of proper
$(d-c+1)$-secant lines to a nondegenerate irreducible projective
variety $X \subset \mathbb{P}^r$ of codimension $c$ and degree $d$.
To this aim, we will investigate the \textit{extremal secant locus}
$\Sigma (X)$ of $X$, that is, the closure of the set of all proper
$(d-c+1)$-secant lines of $X$ in the Grassmannian
$\mathbb{G}(1,\mathbb{P}^r)$.

To give precise statements, we require some notation and
definitions. We first fix a few notations, which we shall keep for
the rest of our paper.

\begin{notation and reminder}\label{4.1'' Notation and Reminder}
Let $X \subset \mathbb{P}^r$ be as above.
\begin{itemize}
\item[\rm{(A)}] Let $\Sigma_m(X)$ be the locus of all $m$-secant lines of $X$ in $\mathbb{G}(1,\mathbb{P}^r)$. That is,
$$\Sigma_m(X) := \{\mathbb{L} \in \mathbb{G}(1,\mathbb{P}^r) \mid \rm{length}(X \cap \mathbb{L}) \geq m\}.$$
This set is closed in $\mathbb{G}(1,\mathbb{P}^r)$. We also shall
use the notation
\begin{equation*}
\Sigma_{\infty}(X) :=   \{\mathbb{L} \in \mathbb{G}(1,\mathbb{P}^r)
\mid \mathbb{L} \subseteq X\}.
\end{equation*}
Thus we have the inclusion $\Sigma_{\infty}(X) \subseteq
\Sigma_m(X)$ and the equality holds if $X$ is cut out by forms of
degree $< m$. In particular, $\Sigma_{\infty}(X)$ is a closed
subset, too.
\item[\rm{(B)}] The set $\Sigma^{\circ}_m(X) := \Sigma_m(X) \setminus \Sigma_{\infty}(X)$
of all proper $m$-secant lines to $X$ is locally closed in
$\mathbb{G}(1,\mathbb{P}^r)$. We define $\mathfrak{d}_m(X)$ and
$\overline{\mathfrak{d}}_m (X)$ respectively as
\begin{equation*}
\mathfrak{d}_m (X):= \dim ~ \Sigma_{m}(X) ~ \mbox{ and }
~\overline{\mathfrak{d}}_m(X) := \dim ~
\overline{\Sigma^{\circ}_m(X)} ~ (= \dim  \Sigma^{\circ}_m(X) ).
\end{equation*}
\item[\rm{(C)}] By definition, the extremal secant locus $\Sigma (X)$ of $X$ is equal to $\overline{\Sigma^{\circ}_{d-c+1} (X)}$.
\item[\rm{(D)}] Recall that in the introduction we define $X$ to be a \textit{variety of maximal sectional regularity} if the general linear curve section of $X$ is of maximal regularity. To be precise, $X$ is of maximal sectional regularity if there exists a nonempty open subset $\mathcal{U} \subset \mathbb{G}(c+1, \mathbb{P}^r )$ such that
\smallskip

\begin{itemize}
\item[$(*)$] For any $\Lambda \in \mathcal{U}$, the intersection
\begin{equation*}
\mathcal{C}_{\Lambda} := X \cap \Lambda \quad \subset \quad \Lambda
= \mathbb{P}^{c+1}
\end{equation*}
is an integral curve of maximal regularity.
\end{itemize}
\smallskip

\noindent In this case, we will denote by $\mathcal{U}(X)$ the
largest open subset of $\mathbb{G}(c+1, \mathbb{P}^r )$ satisfying
the property $(*)$.
\end{itemize}
\end{notation and reminder}

Now, we are heading for the main result of this section. We begin
with the following auxiliary result.

\begin{lemma}\label{2.2'' Lemma}
Let $T$ be an integral closed subset of $\mathbb{G}(1,\mathbb{P}^r)$
and let $\mathbb{H} = \mathbb{P}^{r-1} \subset \mathbb{P}^r$ be a
general hyperplane. Then, the following statements hold.
\begin{itemize}
\item[\rm{(a)}] If $\dim T \leq 1$, then $T \cap \mathbb{G}(1,\mathbb{H}) = \emptyset$.
\item[\rm{(b)}] If $\dim T \geq 2$, then each irreducible component $W$ of $T \cap \mathbb{G}(1,\mathbb{H})$ satisfies
                   $$\dim W = \dim T  - 2.$$
\end{itemize}
\end{lemma}

\begin{proof}
Fix a hyperplane $\mathbb{H}_0 \subset \mathbb{P}^r$. The canonical
action of the integral algebraic group scheme $G =: {\rm
Aut}(\mathbb{P}^r)$ on the integral scheme $X =
\mathbb{G}(1,\mathbb{P}^r)$ is transitive. Thus a result of Kleiman
(see \cite[Corollary 4]{K}) says that all irreducible components of
$g(\mathbb{G}(1,\mathbb{H}_0)) \cap T$ have dimension
$$\dim \mathbb{G}(1,\mathbb{H}_0) +\dim T -\dim
\mathbb{G}(1,\mathbb{P}^r)= \dim T  - 2$$ for general $g \in G$.
\end{proof}

\begin{proposition}\label{prop:dimensionextremal}
Let $X \subset \mathbb{P}^r$ be a nondegenerate irreducible
projective variety of dimension $n$, codimension $c \geq 3$  and
degree $d \geq c+3$. Then
\begin{itemize}
\item[\rm{(a)}] $\mathfrak{d}_{\infty}(X) \leq 2n-3$.
\item[\rm{(b)}] $\mathfrak{d}_{d-c+1} (X) \leq 2n-2$ and the equality is attained if and only if
$X$ is a variety of maximal sectional regularity.
\end{itemize}
\end{proposition}

\begin{proof}
For any hyperplane $\mathbb{H}$ of $\mathbb{P}^r$ and any $m \in
\mathbb{N} \cup \{ \infty \}$, it holds that
\begin{equation}\label{eq:hyperplanesection}
\Sigma_m (X \cap \mathbb{H}) = \Sigma_m (X) \cap \mathbb{G} (1,
\mathbb{H}).
\end{equation}
Now, let $\mathbb{H}_1 , \ldots , \mathbb{H}_{n-1}$ be general
hyperplanes and let $\mathcal{C}:= X \cap \mathbb{H}_1 \cap \cdots
\cap \mathbb{H}_{n-1}$. Then by combining Lemma \ref{2.2'' Lemma}
and (\ref{eq:hyperplanesection}), one can see that
\begin{itemize}
\item[(i)] $\Sigma_m (\mathcal{C})$ is empty if and only if $\mathfrak{d}_m \leq 2n-3$
\end{itemize}
and
\begin{itemize}
\item[(ii)] $\mathfrak{d}_m (\mathcal{C}) \geq 0$ if and only if $\mathfrak{d}_m (X) \geq 2n-2$. In this case, it holds that
    \begin{equation*}
    \mathfrak{d}_m (\mathcal{C}) =  \mathfrak{d}_m (X) - (2n-2).
    \end{equation*}
\end{itemize}

\noindent (a): Obviously, $\Sigma_\infty (\mathcal{C})$ is empty.
Thus we have $\mathfrak{d}_{\infty} (X) \leq 2n-3$ by (i).

\noindent (b): We know that $\mathcal{C}$ has at most one
$(d-c+1)$-secant line and so $\mathfrak{d}_{d-c+1} (\mathcal{C})
\leq 0$ (cf. Notation and Remarks \ref{nar:curveofmaxreg}.(B)). It
follows by (i) and (ii) that $\mathfrak{d}_{d-c+1} (X) \leq  2n-2$.
Moreover, $\mathfrak{d}_{d-c+1} (X) =  2n-2$ if and only if
$\mathfrak{d}_{d-c+1} (\mathcal{C}) = 0$ and hence $\mathcal{C}$ is
a curve of maximal regularity.
\end{proof}

\begin{theorem}\label{thm:maximaldimension}
Let $X \subset \mathbb{P}^r$ be as in Proposition
\ref{prop:dimensionextremal}. Then
$\overline{\mathfrak{d}}_{d-c+1}(X) \leq 2n-2$ and equality is
attained if and only if $X$ is a variety of maximal sectional
regularity.
\end{theorem}

\begin{proof}
Since it holds always that $\mathfrak{d}_{d-c+1} (X) \geq
\overline{\mathfrak{d}}_{d-c+1} (X)$, we get the desired inequality
$\overline{\mathfrak{d}}_{d-c+1} (X) \leq 2n-2$ from Proposition
\ref{prop:dimensionextremal}(b). Also since $\mathfrak{d}_{\infty}
(X) \leq 2n-3$ by Proposition \ref{prop:dimensionextremal}(a), it
holds that $\mathfrak{d}_{d-c+1} (X) = 2n-2$ if and only if
$\overline{\mathfrak{d}}_{d-c+1} (X) = 2n-2$.
\end{proof}

\section{Sectionally Rational Varieties}

\noindent Let $X \subset \mathbb{P}^r$ be a nondegenerate
irreducible projective variety. We will say that $X$ is a
\textit{sectionally rational variety} (resp. \textit{sectionally
smooth rational variety}) if its general linear curve section is
rational (resp. smooth rational). We are interested in this kind of
varieties since any variety of maximal sectional regularity is
sectionally smooth rational (cf. Notation and Remarks
\ref{nar:curveofmaxreg}(A)). The aim of this section is to show --
under some mild conditions -- that a sectionally rational variety is
always obtained as a birational linear projection of a variety of
minimal degree.

\begin{theorem}\label{thm:sectionallyrational}
Let $X \subset \mathbb{P}^r$ be a nondegenerate irreducible
projective variety of dimension $n$ and degree $d$. Assume that
either
\begin{itemize}
\item[\rm{(1)}] $\rm{char}(\Bbbk) = 0$ and $X$ is a sectionally rational variety, or else
\item[\rm{(2)}] $X$ is a sectionally smooth rational surface.
\end{itemize}
Then $X$ is a projection of a variety of minimal degree. More
precisely, $X = \pi_{\Lambda} (\widetilde{X})$ where
\begin{itemize}
\item[\rm{(a)}] $\widetilde{X} \subset \mathbb{P}^{d+n-1}$ is an $n$-dimensional variety of minimal degree,
\item[\rm{(b)}] $\Lambda = \mathbb{P}^{d+n-r-2} \subset \mathbb{P}^{d+n-1}$ is a subspace such that $\widetilde{X} \cap \Lambda = \emptyset$,
\item[\rm{(c)}] $\pi_{\Lambda} : \mathbb{P}^{d+n-1} \setminus \Lambda \rightarrow \mathbb{P}^r$ is the linear projection map from $\Lambda$ and
\item[\rm{(d)}] $\pi_{\Lambda} \upharpoonright : \widetilde{X} \rightarrow X$ is the normalization of $X$.
\end{itemize}
Furthermore, $X$ is a sectionally smooth rational variety if and
only if the singular locus
\begin{equation*}
\rm{Sing}(\pi_{\Lambda}):= \{ x \in X ~|~ \rm{length}(\pi_{\Lambda}
^{-1} (x)) \geq 2 \}
\end{equation*}
of $\pi_{\Lambda} : \widetilde{X} \rightarrow X$ has dimension at
most $n-2$.
\end{theorem}

\begin{proof}
Let $\nu: Y \rightarrow X$ be the normalization of $X$, so that $Y$
is an $n$-dimensional normal projective variety and $\nu$ is a
finite surjective birational morphism. Also the line bundle
$\mathcal{L} := \nu^{*} \mathcal{O}_X(1)$ on $Y$ is ample and base
point free. Let $\mathbb{H}_1,\ldots, \mathbb{H}_{n-1}$ be general
hyperplanes and consider the $\ell$-dimensional irreducible
varieties $X_{\ell} := X \cap \mathbb{H}_1 \cap \cdots \cap
\mathbb{H}_{n-\ell}$ and their preimages $Y_{\ell} :=
\nu^{-1}(X_{\ell}) \quad (\ell=1,\ldots n)$. As the hyperplanes
$\mathbb{H}_j$ are general, $X_{\ell}$ is not contained in the
singular locus of $\nu$ and so $Y_{\ell}$ are irreducible and the
induced finite morphisms
\begin{equation*}
\nu_{\ell} := \nu \upharpoonright : Y_{\ell} \twoheadrightarrow
X_{\ell}, \quad \ell = 1, \ldots, n
\end{equation*}
are birational. Note that $X_1$ is a rational curve since $X$ is
sectionally rational.

Assume first that $\rm{char}(\Bbbk) = 0$. As $Y_{\ell} \subset Y$ is
cut out by the $n-\ell$ general divisors $\nu^{*}(X \cap
\mathbb{H}_1), \ldots ,\nu^{*}(X \cap \mathbb{H}_{n-\ell}) \in
|\mathcal{L}|$, it is normal by the Bertini Theorem \cite[Corollary
3.4.2]{FlOV}. So, the sequence
\begin{equation*}
Y_1 \subset Y_2 \subset \ldots \subset Y_t = Y
\end{equation*}
forms a ladder with normal rungs of the polarized variety
$(Y,\mathcal{L})$ in the sense of T. Fujita\cite{Fu}. As $\nu_1:
Y_1\twoheadrightarrow X_1$ is birational, it follows that $Y_1 \cong
\mathbb{P}^1$ and hence the sectional genus $g(Y,\mathcal{L})$ of
the polarized variety $(Y,\mathcal{L})$ is equal to zero.

Assume now, that $X$ is a sectionally smooth rational surface. Thus
the curve $X_1 \subset X$ is smooth rational and $\nu_1:Y_1
\twoheadrightarrow X_1$ is an isomorphism. Hence, again the
polarized surface $(Y,\mathcal{L})$ has a ladder $Y_1 \subset Y_2 =
Y$ with normal rungs and its sectional genus $g(Y,\mathcal{L})$ is
zero.

Thus, in both cases the sectional genera satisfy $g(Y,\mathcal{L}) =
0$. Therefore, by \cite[Proposition 3.4]{Fu}, the $\Delta$-genus
$\Delta(Y,\mathcal{L})$ of the polarized variety $(Y,\mathcal{L})$
is equal to zero. According to T. Fujita's Classification Theorem
\cite[Theorem 5.15]{Fu}, it now follows that $\mathcal{L}$ is a very
ample line bundle which embeds $Y$ to the $(d+n-1)$-dimensional
projective space as a variety of minimal degree. Let
\begin{equation*}
\widetilde{X} \subset \mathbb{P}^{d+n-1}
\end{equation*}
be the image of the linearly normal embedding
$\varphi_{|\mathcal{L}|} : Y \rightarrow \mathbb{P}^{d+n-1}$. It is
clear that the normalization map $\nu : Y \rightarrow X$ consists of
the embedding $\varphi_{|\mathcal{L}|}$ of $Y$ followed by a linear
projection $\pi_{\Lambda} : \mathbb{P}^{d+n-1} \setminus \Lambda
\rightarrow \mathbb{P}^r$ from a linear space $\Lambda =
\mathbb{P}^{d+n-r-2}$. In particular, the map $\pi_{\Lambda} :
\widetilde{X} \rightarrow X$ is the normalization of $X$.

Finally, consider the short exact sequence $0 \rightarrow
\mathcal{O}_X \rightarrow (\pi_{\Lambda} )_*
\mathcal{O}_{\widetilde{X}} \rightarrow \mathcal{F}  \rightarrow 0$
on $X$ where $\mathcal{F}$ is the quotient sheaf $(\pi_{\Lambda} )_*
\mathcal{O}_{\widetilde{X}}/\mathcal{O}_X$. Note that
$\rm{Sing}(\pi_{\Lambda}) = \rm{Supp}(\mathcal{F})$. In particular,
the dimension of $\rm{Sing}(\pi_{\Lambda})$ is equal to the degree
of the Euler-Poincar$\acute{e}$ characteristic $\chi
(\mathcal{F}(t))$. Let us write $\chi (\mathcal{O}_X (t))$ and $\chi
(\mathcal{O}_{\widetilde{X}} (t))$ respectively as
\begin{equation*}
\chi (\mathcal{O}_X (t)) = \sum_{j=0} ^n \chi_j (\mathcal{O}_X (1))
{{t+j-1} \choose {j}} \quad \mbox{and} \quad \chi
(\mathcal{O}_{\widetilde{X}} (t)) = \sum_{j=0} ^n \chi_j
(\mathcal{O}_{\widetilde{X}} (1)) {{t+j-1} \choose {j}}.
\end{equation*}
Here, it holds that $\chi_n (\mathcal{O}_X (1))=\chi_n
(\mathcal{O}_{\widetilde{X}} (1))=d$ since $\pi_{\Lambda} :
\widetilde{X} \rightarrow X$ is a finite birational morphism. Also
$\chi_{n-1} (\mathcal{O}_{\widetilde{X}} (1))=1$ since the general
linear curve section of $\widetilde{X}$ is a smooth rational curve.
Now, let $m$ be the degree of the polynomial $\chi
(\mathcal{F}(t))$. Then, from the relation $\chi (\mathcal{F}(t))
=\chi (\mathcal{O}_{\widetilde{X}} (t)) - \chi (\mathcal{O}_X (t))$
among the Euler-Poincar$\acute{e}$ characteristics, we can see that
$m \leq n-2$ if and only if $\chi_{n-1} (\mathcal{O}_X (1))=1$ and
hence the general linear curve section of $X$ is a curve of
arithmetic genus $0$, or equivalently, a smooth rational curve.
\end{proof}

It occurs to me, that Corollary 4.2 could be extended as follows:\\

Assume that $X \subset \mathbb{P}^r$ is as in Theorem 4.1 and that X has only finitely many non-normal points (which is always the case if condition (2) of Theorem 4.1 holds). Then
$$\chi(\mathcal{O}_X(t)) = d\binom{t+n-1}{n} + \binom{t+n-1}{n-1} - \delta(X)$$
where
$$\delta(X) = \mathrm{length}\big((\pi_{\Lambda})_{*}\mathcal{O}_{\widetilde{X}}/ \mathcal{O}_X\big).$$

\begin{corollary}\label{cor:singularityofsurface}
Let $X \subset \mathbb{P}^r$ be a sectionally smooth rational
surface of degree $d$. Then
\begin{equation*}
\chi (\mathcal{O}_X (t)) = d {{t+1} \choose {2}}  + t + 1- \delta
(X)
\end{equation*}
for a non-negative integer $\delta (X)$. Furthermore, $X$ is smooth
if and only if $\delta (X)=0$.
\end{corollary}

\begin{proof}
From Theorem \ref{thm:sectionallyrational} and its proof, we have
$\chi (\mathcal{O}_X (t)) =\chi (\mathcal{O}_{\widetilde{X}} (t)) -
\chi (\mathcal{F}(t))$. Also $\rm{Supp} (\mathcal{F})$ is at most a
finite set and hence $\chi (\mathcal{F}(t))$ is a non-negative
integer, say $\delta (X)$. Since
\begin{equation*}
\chi (\mathcal{O}_{\widetilde{X}} (t)) = d {{t+1} \choose {2}}  + t
+ 1,
\end{equation*}
we get the desired formula of $\chi (\mathcal{O}_X (t))$. Moreover,
$\delta (X)=0$ if and only if $\pi_{\Lambda} : \widetilde{X}
\rightarrow \mathbb{P}^r$ is an isomorphic projection of
$\widetilde{X}$. Since $\widetilde{X}$ is either smooth or else a
cone, we can rephrase this fact as $\delta (X)=0$ if and only if $X$
is smooth.
\end{proof}

\begin{corollary}\label{cor:linearprojection}
Let $X \subset \mathbb{P}^r$ be a nondegenerate irreducible
projective variety of dimension $n$ and degree $d$ which is of
maximal sectional regularity. If $\rm{char}(\Bbbk) = 0$ or if $n=2$,
then $X$ is a projection of a rational normal scroll. More
precisely, $X = \pi_{\Lambda} (\widetilde{X})$ where
\begin{itemize}
\item[\rm{(a)}] $\widetilde{X} \subset \mathbb{P}^{d+n-1}$ is a rational normal $n$-fold scroll,
\item[\rm{(b)}] $\Lambda = \mathbb{P}^{d+n-r-2} \subset \mathbb{P}^{d+n-1}$ is a subspace with $\widetilde{X}
\cap \Lambda = \emptyset$,
\item[\rm{(c)}] $\pi_{\Lambda} : \mathbb{P}^{d+n-1} \setminus \Lambda \rightarrow \mathbb{P}^r$ is the linear projection map from $\Lambda$ and
\item[\rm{(d)}] $\pi_{\Lambda} \upharpoonright : \widetilde{X} \rightarrow X$ is the normalization of $X$.
\end{itemize}
Furthermore, if $X$ is not a cone then $\widetilde{X}$ is a smooth
rational normal scroll.
\end{corollary}

\begin{proof}
Since $X$ is sectionally smooth rational (cf. Notation and Remarks
\ref{nar:curveofmaxreg}(A)), it holds by Theorem
\ref{thm:sectionallyrational} that $X$ is a projection of a variety
of minimal degree. In our case, $d$ is at least $5$ and hence the
projecting variety $\widetilde{X}$ should be an $n$-fold rational
normal scroll.

If $X$ is not a cone, then $\widetilde{X}$ should be not a cone and
hence smooth by the well-known classification result of varieties of
minimal degree.
\end{proof}

\section{The extremal variety of a variety of maximal sectional regularity}

\noindent Let $X \subset \mathbb{P}^r$ be a nondegenerate
irreducible projective variety of dimension $n$, codimension $c \geq
3$ and degree $d \geq c+3$. Assume that $X$ is of maximal sectional
regularity and let $\mathcal{U}(X) \subset
\mathbb{G}(c+1,\mathbb{P}^r)$ be as in Notation and Reminder
\ref{4.1'' Notation and Reminder}(D). Then for any $\Lambda \in
\mathcal{U}(X)$, the intersection
\begin{equation*}
\mathcal{C}_{\Lambda} := X \cap \Lambda \subset \mathbb{P}^{c+1}
\end{equation*}
is an integral curve of maximal regularity. In particular, it admits
a unique $(d-c+1)$-secant line, say $\mathbb{L}_{\Lambda}$ (cf.
Notation and Remarks \ref{nar:curveofmaxreg}(B)). Along this line,
we consider the \textit{extremal variety} $\mathbb{F}(X)$ of $X$
which is defined as
\begin{equation*}
\mathbb{F}(X) = \overline{\bigcup_{\Lambda \in \mathcal{U}(X)}
\mathbb{L}_{\Lambda}} ~ \subset ~ \mathbb{P}^r .
\end{equation*}
Through the next two sections it will turn out that either
$\mathbb{F}(X)$ is an $n$-dimensional linear space or else $c=3$ and
$\mathbb{F}(X)$ is the $(n+1)$-fold rational normal scroll
\begin{equation*}
S(\underbrace{0,\ldots,0}_{(n-2)-\rm{times}},1,1,1) \subset
\mathbb{P}^{n+3}.
\end{equation*}
This structure of the extremal variety will play a crucial role in
the classification of varieties of maximal sectional regularity.

Along this line, this section is devoted to prove a criterion on the
linearity of $\mathbb{F}(X)$ and to classify varieties of maximal
sectional regularity whose extremal variety is a linear space.

\begin{lemma}\label{lem:planaritycriterion}
Let $X \subset \mathbb{P}^r$ be a nondegenerate irreducible
projective variety of dimension $n$, codimension $c \geq 3$ and
degree $d \geq c+3$. Suppose that there exists an $n$-dimensional
linear subspace $\mathbb{F} \subset \mathbb{P}^r$ such that $\dim(X
\cap \mathbb{F}) = n-1$.
\begin{itemize}
\item[\rm{(a)}] If $\deg_{\mathbb{F}}(X\cap \mathbb{F}) \geq d-c+1$, then $X$ is of maximal sectional regularity and $\mathbb{F} = \mathbb{F}(X)$.
\item[\rm{(b)}] Suppose that $X$ is a variety of maximal sectional regularity. If $\mathbb{L}_{\Lambda} \subset \mathbb{F}$ for general $\Lambda \in \mathcal{U}(X)$, then $\mathbb{F}(X) = \mathbb{F}$.
\end{itemize}
\end{lemma}

\begin{proof}
(a): Set $t := \deg_{\mathbb{F}}(X \cap \mathbb{F})$ and let
$\Lambda = \mathbb{P}^{c+1} \in \mathbb{G} (c+1,\mathbb{P}^r)$ be a
general member. Then the line $\mathbb{L} := \mathbb{F} \cap
\Lambda$ is $t$-secant to the integral curve $\mathcal{C}_{\Lambda}
:= X \cap \Lambda \subset \mathbb{P}^{c+1}$ of degree $d$. Therefore
$t \leq \reg(\mathcal{C}_{\Lambda} ) \leq d-c+1$, whence $t =
d-c+1$. Thus $\mathcal{C}_\Lambda \subset \mathbb{P}^{c+1}$ is a
curve of maximal regularity and $\mathbb{L} \subset \mathbb{F}$ is
its unique $(d-c+1)$-secant line. This shows that $X$ is a variety
of maximal sectional regularity and $\mathbb{F} \subseteq
\mathbb{F}(X)$. Now, let $\Lambda \in \mathcal{U}(X)$ and consider
the integral curve $\mathcal{C}_{\Lambda} := X \cap \Lambda \subset
\mathbb{P}^{c+1}$ of maximal regularity. It is clear that
$\mathbb{L} := \mathbb{F} \cap \Lambda$ is a line such that
$\rm{length}(\mathcal{C}_\Lambda \cap \mathbb{L}) \geq d-c+1$.
Thus $\mathbb{L}$ is the unique $(d-c+1)$-secant line to $\mathcal{C}_\Lambda$. This shows that $\mathbb{F} = \mathbb{F}(X)$.\\
(b): For general $\Lambda \in \mathcal{U}(X)$, we have
$\mathbb{L}_\Lambda := \mathbb{F} \cap \Lambda$ and thus
\begin{equation*}
d-c+1 = \rm{length}(\mathcal{C}_\Lambda \cap \mathbb{L}_\Lambda)
\leq \rm{length}(X \cap \mathbb{L}_\Lambda) = \rm{length}\big((X
\cap \mathbb{F})\cap \mathbb{L}_\Lambda \big) = \deg_{\mathbb{F}} (X
\cap \mathbb{F}).
\end{equation*}
So, our claim follows by statement (a).
\end{proof}

\begin{proposition}\label{prop:planarity}
Let $X \subset \mathbb{P}^r$ be a nondegenerate irreducible
projective variety of dimension $n$, codimension $c \geq 3$ and
degree $d \geq c+3$. If $X$ is a variety of maximal sectional
regularity, then the following conditions are equivalent.
\begin{itemize}
\item[\rm{(i)}] $\mathbb{F}(X)$ is an $n$-dimensional linear space.
\item[\rm{(ii)}] $\dim \mathbb{F}(X)  = n$.
\end{itemize}
\end{proposition}

\begin{proof}
(i) $\Rightarrow$ (ii) is obvious.\\
(ii) $\Rightarrow$ (i): Let $D_1,\ldots, D_t$ be the different
$n$-dimensional irreducible components of $\mathbb{F}(X)$ and write
\begin{equation*}
\mathbb{F}(X) = D_1 \cup \cdots \cup D_t \cup E
\end{equation*}
where $E \subset \mathbb{P}^r$ is a closed subset of dimension at most $n-1$. Also we write $X \cap D_j$ as $V_j \cup W_j$ where $V_j$ is
either empty of else an equidimensional scheme of dimension $n-1$
and $W_j$ is a closed subscheme of dimension at most $n-2$. Now, choose a
general $\Lambda \in \mathcal{U}(X)$. So, it avoids $W_1 , \ldots ,
W_t$ and $E \cap \Lambda$ is at most a finite set. Then we have
\begin{equation*}
\mathbb{L}_\Lambda \subseteq \mathbb{F}(X) \cap \Lambda = (D_1 \cap \Lambda) \cup \cdots \cup (D_t \cap \Lambda) \cup (E \cap \Lambda)
\end{equation*}
and hence $\mathbb{L}_\Lambda = D_i \cap \Lambda$ for some $i$. This
means that $D_i$ is a linear space. Also
\begin{equation*}
\mathcal{C}_{\Lambda} \cap \mathbb{L}_\Lambda = (X \cap D_i) \cap
\Lambda = V_i \cap \mathbb{L}_\Lambda
\end{equation*}
is nonempty and hence $X \cap D_j$ is of dimension $n-1$.
Furthermore, we have
\begin{equation*}
\deg_{D_i}( X \cap D_i ) = \deg_{D_i} V_i = \mbox{length} (X \cap
D_i \cap \Lambda' ) = \mbox{length} (V_i \cap
\mathbb{L}_\Lambda)=d-c+1.
\end{equation*}
Therefore it follows by Lemma \ref{lem:planaritycriterion}(a) that
$\mathbb{F} (X)$ coincides with $D_i$.
\end{proof}

\begin{lemma}\label{lem:elementarycrucial}
Let $\widetilde{X} = S(\underbrace{0, \ldots , 0}_{k-\rm{times}} ,
a_{k+1} , \ldots , a_n) \subset \mathbb{P}^{d+n-1}$ be a rational
normal $n$-fold scroll for some $0 \leq k \leq n-2$ and positive
integers $a_{k+1} \leq \ldots \leq a_n$. Let $D \subset
\widetilde{X}$ be a divisor linearly equivalent to $sH + tF$ where
$H$ is a hyperplane section of $\widetilde{X}$ and $F \subset
\widetilde{X}$ is an $(n-1)$-dimensional linear space. Then
\begin{itemize}
\item[\rm{(a)}] If $s \geq 2$ or $s=1$ and $t>0$ or $s=0$ and $t > a_n$, then $\langle D \rangle
= \mathbb{P}^{d+n-1}$.
\item[\rm{(b)}] If $s=1$ and $t \leq 0$, then $\rm{dim}~\langle D \rangle
= d+n+t-2$.
\item[\rm{(c)}] If $s=0$ and $a_{i-1} < t \leq a_i$ for some $i\leq
n$, then
\begin{equation*}
\rm{dim}~\langle D \rangle = (a_{k+1} + \cdots + a_{i-1}) +
(n-i+1)t+(i-2).
\end{equation*}
\item[\rm{(d)}] If $s=1$ and $t \leq 0$, then $I_D = I_{\widetilde{X}} + I_{\langle D
\rangle}$.
\end{itemize}
\end{lemma}

\begin{proof}
First, observe that it is enough to show the statements in the case
where $k=0$. So, we suppose that $\widetilde{X}$ is smooth. Let $R$
denote the homogeneous coordinate ring of $\mathbb{P}^{d+n-1}$. Also
let $I_D$ and $\mathcal{I}_D$ be respectively the homogeneous ideal
and the sheaf of ideals of $D$ in $\mathbb{P}^{d+n-1}$ and consider
the exact sequence $0 \rightarrow \mathcal{I}_{\widetilde{X}}
\rightarrow \mathcal{I}_{D} \rightarrow \mathcal{O}_{\widetilde{X}}
(-D) \rightarrow 0$. Since $\widetilde{X}$ is a projectively normal
variety, we get the exact sequence
\begin{equation}\label{eq:relativeideal}
0 \rightarrow I_{\widetilde{X}} \rightarrow I_{D} \rightarrow E =
\bigoplus_{j \in \mathbb{Z}} E_j \rightarrow 0
\end{equation}
of graded $R$-modules where $E_j := H^0 (\widetilde{X}
,\mathcal{O}_{\widetilde{X}} (jH-D))$.
\smallskip

\noindent $\rm{(a)} \sim \rm{(c)}$: From the above short exact
sequence (\ref{eq:relativeideal}), we know that
\begin{equation*}
\rm{dim}~\langle D \rangle = d+n-1 - h^0 (\widetilde{X}
,\mathcal{O}_{\widetilde{X}} (H-D)).
\end{equation*}
If $s \geq 2$, then $H^0 (\widetilde{X} ,\mathcal{O}_{\widetilde{X}}
(H-D))=0$ and hence $D$ spans the whole ambient space. If $s=1$,
then $H^0 (\widetilde{X} ,\mathcal{O}_{\widetilde{X}} (H-D)) \cong
H^0 (\mathbb{P}^1 , \mathcal{O}_{\mathbb{P}^1} (-t))$ and hence we
get the desired result. If $s=0$, then
\begin{equation*}
H^0 (\widetilde{X} ,\mathcal{O}_{\widetilde{X}} (H-D)) \cong
\bigoplus_{k+1 \leq j \leq n} H^0 (\mathbb{P}^1 ,
\mathcal{O}_{\mathbb{P}^1} (a_j -t))
\end{equation*}
and hence again we get the desired formula.
\smallskip

\noindent $\rm{(d)}$: One can easily check that
$\mathcal{O}_{\widetilde{X}} (H-D)$ is $1$-regular with respect to
$\mathcal{O}_{\widetilde{X}} (H)$ and hence $E$ is generated by
$E_1$. Then it follows from (\ref{eq:relativeideal}) that $I_D$ is
generated by its degree one piece and $I_{\widetilde{X}}$. This
completes the proof.
\end{proof}

\begin{theorem}\label{thm:linearextremalvariety}
Suppose that $\rm{char}(\Bbbk) = 0$ or $n=2$ and let $X \subset
\mathbb{P}^r$ be a nondegenerate irreducible projective variety of
dimension $n \geq 2$, codimension $c \geq 3$ and degree $d \geq
c+3$. Then the following two conditions are equivalent:
\begin{itemize}
\item[\rm{(i)}] $X$ is a variety of maximal sectional regularity and $\mathbb{F}(X)$ is an $n$-dimensional linear space.
\item[\rm{(ii)}] Either $X$ is a cone over a curve of maximal regularity or else $X =\pi_{\Lambda} (\widetilde{X})$ where
\begin{itemize}
\item[1.] $\widetilde{X} = S(\underbrace{0, \ldots , 0}_{k-\rm{times}} , a_{k+1} , \ldots , a_n) \subset \mathbb{P}^{d+n-1}$
is a rational normal $n$-fold scroll for some $0 \leq k \leq n-2$
and positive integers $a_{k+1} \leq \ldots \leq a_n$,
\item[2.] $D \subset \widetilde{X}$ is an effective divisor linearly equivalent to $H + (1-c)F$ where $H$ is a hyperplane section
of $\widetilde{X}$ and $F \subset \widetilde{X}$ is an
$(n-1)$-dimensional linear space $($hence $\langle D \rangle =
\mathbb{P}^{d-c+n-1})$, and
\item[3.] $\Lambda \subset \langle D \rangle$ is a $(d-c-2)$-dimensional subspace such that the restriction $\pi_{\Lambda} \upharpoonright :\langle D \rangle \setminus \Lambda \twoheadrightarrow \mathbb{P}^n$ is generically injective along $D$.
\end{itemize}
\end{itemize}
In this case, $X$ is singular.
\end{theorem}

\begin{proof}
(i) $\Rightarrow$ (ii) : Let $\widetilde{X} \subset
\mathbb{P}^{d+n-1}$ and $\Lambda = \mathbb{P}^{d-c-2}$ be as in
Corollary \ref{cor:linearprojection}. Thus
\begin{equation*}
\widetilde{X} = S(\underbrace{0, \ldots , 0}_{k-\rm{times}} ,
a_{k+1} , \ldots , a_n) \subset \mathbb{P}^{d+n-1}
\end{equation*}
for some $0 \leq k \leq n-1$ and positive integers $a_{k+1} \leq
\ldots \leq a_n$. If $k=n-1$, then $X$ should be a cone over a curve
of maximal regularity. Now, assume that $k<n-1$. Note that $G:=X
\cap \mathbb{F}(X)$ contains a hypersurface of $\mathbb{F}(X)$ whose
degree is at least $d-c+1$. Let $\mathbb{E}$ and $D$ be respectively
the pre-images of $\mathbb{F}(X)$ and $G$ by $\pi_{\Lambda}$. Thus
$\mathbb{E}$ is a $(d-c+n-1)$-dimensional linear space and $D
\subset \mathbb{E}$ is a subscheme of dimension $n-1$ and degree
$\geq d-c+1$. Thus $D$ contains a divisor $D'$ of $\widetilde{X}$
whose degree is at least $d-c+1$. Let $sH+ tF$ be the divisor class
of $D'$ in $\widetilde{X}$. Thus we have the inequality $\deg (D') =
sd+t \geq d-c+1$. Note that $\langle D' \rangle$ is a subspace of
$\mathbb{E}$ which is a proper subspace of $\mathbb{P}^{d+n-1}$. On
the other hand, one can show by Lemma
\ref{lem:elementarycrucial}(a)$\sim$(c) that if $s=0$ and hence $t
\geq d-c+1$ or $s=1$ and $t > 1-c$ or $s \geq 2$ then the dimension
of $\langle D' \rangle$ is strictly greater than that of
$\mathbb{E}$. Therefore we have $s=1$ and $t=1-c$. Then it follows
respectively by Lemma \ref{lem:elementarycrucial}(b) and (d) that
\begin{equation*}
(\dagger) ~ \langle D' \rangle = \mathbb{E} \quad \mbox{and} \quad (\ddagger) ~
\widetilde{X} \cap \langle D' \rangle = D'.
\end{equation*}
From $(\dagger)$, we know that $\langle D \rangle$ contains $\Lambda$.
Also from $(\ddagger)$, it holds that $D=D'$ and hence $G =
\pi_{\Lambda}(D')$ is a hypersurface in $\mathbb{F}(X)$. Finally,
the restriction $\pi_{\Lambda}\upharpoonright: \langle D \rangle
\setminus \Lambda \twoheadrightarrow \mathbb{F}(X)$ is generically
injective along $D$ since $D$ and $G$ have the same degree.

(i) $\Leftarrow$ (ii) : Assume that $X$ is a cone over a curve
$\mathcal{C} \subset \mathbb{P}^{c+1}$ of maximal regularity. Let
$\Delta = \mathbb{P}^{n-2}$ be the vertex of $X$ and
$\mathbb{L}_{\mathcal{C}}$ the unique $(d-c+1)$-secant line of
$\mathcal{C}$. It is clear that $X$ is a variety of maximal
sectional regularity. Also the $n$-dimensional linear space
$\mathbb{F} := \langle \mathbb{L}_{\mathcal{C}},\Delta \rangle$
satisfies the conditions that $\dim (X \cap \mathbb{F})=n-1$ and
$\mathbb{L}_{\Lambda'} \subset \mathbb{F}$ for general $\Lambda' \in
\mathcal{U}(X)$. Therefore it follows by Lemma
\ref{lem:planaritycriterion}(b) that $\mathbb{F} (X) = \mathbb{F}$.

Now, consider the second case. Let $\mathbb{F}$ be the
$n$-dimensional linear space $\pi_{\Lambda}(\mathbb{E} \setminus
\Lambda)$. Note that $D$ is of degree $d-c+1$. As
$\pi_{\Lambda}\upharpoonright : \langle D \rangle \setminus \Lambda
\twoheadrightarrow \mathbb{F}$ is generically injective along $D$,
the map $\pi_{\Lambda}: \widetilde{X} \rightarrow X$ is birational
and hence $G:= \pi_{\Lambda}(D) \subset \mathbb{F}$ is a codimension
one subscheme of degree $d-c+1$. As $G \subset X \cap \mathbb{F}$ it
holds that $X \cap \mathbb{F}$ is of dimension $n-1$ and of degree
$\geq d-c+1$. It follows by Lemma \ref{lem:planaritycriterion} that
$X$ is a variety of maximal sectional regularity and $\mathbb{F}(X)
= \mathbb{F}$ is an $n$-dimensional linear space.

Finally, the map $\pi_{\Lambda}\upharpoonright : D
\twoheadrightarrow G$ cannot be an isomorphism since $G$ in
$\mathbb{P}^n$ is linearly normal. This implies that the finite
birational morphism $\pi_{\Lambda} : \widetilde{X} \rightarrow X$ is
not an isomorphism and hence $X$ is singular.
\end{proof}

\section{Surfaces of Maximal Sectional Regularity}

\noindent This section is aimed to classify projective surfaces of
maximal sectional regularity. To this end, we will first classify
the extremal varieties of surfaces of maximal sectional regularity.
To give precise statements, we require some notation and
definitions.

\begin{notation and remark}\label{notrmk:fourfoldscroll}
Let $r \geq 5$ and let $X \subset \mathbb{P}^r$ be a surface of
degree $d \geq r+1$ and of maximal sectional regularity. Thus
$\mathcal{U}(X)$ is a nonempty open subset of $(\mathbb{P}^r ) ^*$.
For every $\mathbb{H} \in \mathcal{U}(X)$, the intersection
$\mathcal{C}_\mathbb{H} := X \cap \mathbb{H} \subset
\mathbb{P}^{r-1}$ is an integral curve of maximal regularity. We
denote by $\mathbb{L}_\mathbb{H}$ the unique $(d-r+3)$-secant line
to $\mathcal{C}_\mathbb{H}$.
\begin{itemize}
\item[\rm{(A)}] Theorem \ref{thm:sectionallyrational} says that the
normalization $\pi : \widetilde{X} \rightarrow X$ of $X$ is realized
as a linear projection of a rational normal surface scroll
$\widetilde{X} \subset \mathbb{P}^{d+1}$. In particular, $X$ is
covered by lines.
\item[\rm{(B)}] By view of Theorem \ref{thm:sectionallyrational} it follows that the singular locus $\mbox{Sing}(X)$ of $X$ is finite. Therefore the set
\begin{equation*}
\mathcal{V}(X) := \{ \mathbb{H} \in \mathcal{U}(X) ~|~
\mbox{Sing}(X) \cap \mathbb{H} = \emptyset \}
\end{equation*}
is a nonempty open subset of $\mathcal{U}(X)$.
\item[\rm{(C)}] For each $\mathbb{H} \in \mathcal{V}(X)$, consider the
nondegenerate projective surface
\begin{equation*}
Z_\mathbb{H} := \overline{ \pi_\mathbb{H} (X \setminus
\mathbb{L}_\mathbb{H} )} \subset \mathbb{P}^{r-2}
\end{equation*}
where $\pi_\mathbb{H} : \mathbb{P}^r \setminus \mathbb{L}_\mathbb{H}
\rightarrow \mathbb{P}^{r-2}$ is the linear projection from
$\mathbb{L}_\mathbb{H}$. Since the intersection $X \cap
\mathbb{L}_\mathbb{H}$ is contained in the smooth locus of $X$, we
have
\begin{equation*}
\deg (Z_\mathbb{H} ) = \deg (X) - \mbox{length} (X \cap
\mathbb{L}_\mathbb{H} ) = d - (d-r+3) = r-3
\end{equation*}
and hence $Z_\mathbb{H}$ is a surface of minimal degree. By (A),
$Z_\mathbb{H}$ is covered by lines and so it is a rational normal
surface scroll. Write
\begin{equation*}
Z_\mathbb{H} = S(b_\mathbb{H} , r-3-b_\mathbb{H} )
\end{equation*}
for some $0 \leq b_\mathbb{H} \leq \frac{r-3}{2}$. Whence the join
\begin{equation*}
W_\mathbb{H} := \mbox{Join} (\mathbb{L}_\mathbb{H} , X) =
\mbox{Join} (\mathbb{L}_\mathbb{H} , Z_\mathbb{H} )
\end{equation*}
is a fourfold rational normal scroll $S(0,0,b_\mathbb{H} ,
r-3-b_\mathbb{H} )$ such that $S(0,0)=\mathbb{L}_\mathbb{H}$ and
$S(b_\mathbb{H} , r-3-b_\mathbb{H}) = Z_\mathbb{H}$.
\item[\rm{(D)}] Recall that $W_\mathbb{H}$ is cut out by quadrics. As $X$ is a subset of $W_\mathbb{H}$ and
    \begin{equation*}
    \rm{length} (X\cap \mathbb{L}_{\mathbb{H}'}) = d-r+3 > 2
    \end{equation*}
    for any $\mathbb{H}' \in \mathcal{U}(X)$, we have $\mathbb{F}(X)
\subset W_\mathbb{H}$.
\end{itemize}
\end{notation and remark}

\begin{proposition}\label{prop:structureofF(X)}
Let $5 \leq r < d$ and let $X \subset \mathbb{P}^r$ be a surface of
degree $d$ and of maximal sectional regularity. Let the notations be
as in Notation and Remark \ref{notrmk:fourfoldscroll}.
\begin{itemize}
\item[\rm{(a)}] Suppose that $b_\mathbb{H} = 0$ for some $\mathbb{H} \in \mathcal{V}(X)$. Then the following statements hold:
\begin{itemize}
\item[\rm{(1)}] $\mathbb{F}(X)$ is equal to the set of vertices of $W_\mathbb{H}$. In particular, it is a
plane;
\item[\rm{(2)}] For any $\mathbb{H}' \in \mathcal{V}(X)$, $W_{\mathbb{H}'} = W_\mathbb{H}$ and hence $b_{\mathbb{H}'} = 0$.
\end{itemize}
\item[\rm{(b)}] Suppose that $b_\mathbb{H} > 0$ for some $\mathbb{H} \in \mathcal{V}(X)$. Then it holds
\begin{itemize}
\item[\rm{(1)}] $\mathbb{L}_{\mathbb{H}'}$ is either equal or disjoint to
$\mathbb{L}_\mathbb{H}$ for any $\mathbb{H}' \in \mathcal{V}(X)$;
\item[\rm{(2)}] $\dim ~ \mathbb{F}(X) = 3$ and $X \subset
\mathbb{F}(X)$;
\item[\rm{(3)}] $r = 5$;
\item[\rm{(4)}] $b_{\mathbb{H}'} = 1$ for all $\mathbb{H}' \in \mathcal{V}(X)$.
\end{itemize}
\end{itemize}
\end{proposition}

\begin{proof}
\rm{(a)} Let $\mathbb{F}$ be the plane $S(0,0,0)$ of vertices of
$W_\mathbb{H}$. Thus we have $W_\mathbb{H} =
\mbox{Join}(\mathbb{F},X)$. For a general member $\mathbb{H}' \in
\mathcal{V}(X)$, we have
\begin{equation}\label{eq:embeddingscrollofcurve}
\mathcal{C}_{\mathbb{H}'} \subset W_\mathbb{H} \cap \mathbb{H}' =
S(0,0,r-3)
\end{equation}
where the line $S(0,0)$ in $W_\mathbb{H} \cap \mathbb{H}'$ is equal
to the intersection $\mathbb{F} \cap \mathbb{H}'$. By Proposition
\ref{prop:2.3}, this line is the $(d-r+3)$-secant line
$\mathbb{L}_{\mathbb{H}'}$ of $\mathcal{C}_{\mathbb{H}'}$ and hence
$\mathbb{L}_{\mathbb{H}'}$ is contained in $\mathbb{F}$.
Furthermore, $X \cap \mathbb{F}$ is of dimension one. Thus,
$\mathbb{F}(X)$ is exactly equal to the plane $\mathbb{F}$ by Lemma
\ref{lem:planaritycriterion}. Moreover the above
$(\ref{eq:embeddingscrollofcurve})$ holds for all $\mathbb{H}' \in
\mathcal{V}(X)$. This implies that
\begin{equation*}
W_{\mathbb{H}'} = \mbox{Join}(\mathbb{L}_{\mathbb{H}'} , X) \subset
\mbox{Join}(\mathbb{F},X) = W_\mathbb{H} .
\end{equation*}
Therefore we have $W_{\mathbb{H}'} = W_\mathbb{H}$ and
$b_{\mathbb{H}'} = 0$.
\smallskip

\noindent \rm{(b)} Suppose that $b_\mathbb{H} > 0$ for some
$\mathbb{H} \in \mathcal{V}(X)$. Let $f : X \setminus
\mathbb{L}_\mathbb{H} \rightarrow Z_\mathbb{H}$ be the restriction
of the linear projection map from $\mathbb{L}_\mathbb{H}$. By
Notation and Remark \ref{notrmk:fourfoldscroll}.(A) and (C), we know
that $f$ is a birational map, $X$ is covered by lines and those
lines map to the ruling lines of the smooth rational normal surface
scroll $Z_\mathbb{H}$.

Let $\mathbb{H}' \in \mathcal{V}(X)$. Then we claim that the lines
$\mathbb{L}_{\mathbb{H}'}$ and $\mathbb{L}_\mathbb{H}$ are either
disjoint or else $\mathbb{L}_{\mathbb{H}'} = \mathbb{L}_\mathbb{H}$.
Consider the curve $\mathcal{C}' :=
\overline{f(\mathcal{C}_{\mathbb{H}'} \setminus
{\mathbb{L}_\mathbb{H}})}$. The map $f \upharpoonright :
\mathcal{C}_{\mathbb{H}'} \setminus \mathbb{L}_\mathbb{H}
\rightarrow \mathcal{C}'$ extends to a unique birational morphism $g
: \mathcal{C}_{\mathbb{H}'} \rightarrow \mathcal{C}'$ since
$\mathcal{C}_{\mathbb{H}'}$ is smooth. Also $\mathcal{C}'$ and a
general ruling line of $Z_\mathbb{H}$ intersect at a point. This
means that $\mathcal{C}'$ is a section of $Z_\mathbb{H}$ and hence
it is a smooth rational curve. So, the morphism $g$ is indeed an
isomorphism. Assume that $\mathbb{L}_\mathbb{H}$ and
$\mathbb{L}_{\mathbb{H}'}$ are different and not disjoint. Then they
must meet at the vertex $q$ of the threefold scroll $W_\mathbb{H}
\cap \mathbb{H}' = S(0,b_\mathbb{H},r-3-b_\mathbb{H})$. Hence, the
point $p = \pi_\mathbb{H} (\mathbb{L}_{\mathbb{H}'}\setminus\{q\})
\in Z_\mathbb{H}$ satisfies
\begin{equation*}
p \in C' \quad \mbox{and} \quad \rm{length} \big( g^{-1}(p) \big) =
\mbox{length} (\mathbb{L}_{\mathbb{H}'}\cap C_{\mathbb{H}'}) = d-r+3
> 1,
\end{equation*}
which is a contradiction. This proves the stated disjointness of
$\mathbb{L}_{\mathbb{H}'}$ and $\mathbb{L}_{\mathbb{H}}$.

The previous disjointness implies that $\mathbb{F}(X)$ cannot be a
plane since it should contain two disjoint lines. It follows by
Proposition \ref{prop:planarity} that the dimension of
$\mathbb{F}(X)$ is at least $3$. On the other hand, we will show
that $\mathbb{F}(X)$ has dimension at most $3$. To do so, let us
consider the incidence correspondence
\begin{equation*}
\Sigma := \{(x,\mathbb{L}) \in \mathbb{P}^r \times
\mathbb{G}(1,\mathbb{P}^r) \mid x \in \mathbb{L} \} \subseteq
\mathbb{P}^r \times \mathbb{G}(1,\mathbb{P}^r)
\end{equation*}
and the canonical projections
\begin{equation*}
\varphi : \Sigma \rightarrow \mathbb{P}^r \mbox{ and } \psi : \Sigma
\rightarrow \mathbb{G}(1,\mathbb{P}^r).
\end{equation*}
Theorem \ref{thm:maximaldimension} says that
$\overline{\Sigma_{d-c+1} ^o (X)}$ is $2$-dimensional and so
$\psi^{-1} \big( \overline{\Sigma_{d-c+1} ^o (X)} \big)$ is
$3$-dimensional. Therefore the closed subset
\begin{equation*}
\varphi \big( \psi^{-1} \big( \overline{\Sigma_{d-c+1} ^o (X)} \big)
\big) = \bigcup_{\mathbb{L} \in  \overline{\Sigma_{d-c+1} ^o (X)} }
\mathbb{L} \quad \subset \mathbb{P}^r
\end{equation*}
has dimension $\leq 3$. Since $\mathbb{F} (X)$ is contained in this
closed subset, we have $\dim ~ \mathbb{F}(X) \leq 3$. In
consequence, it is shown that
\begin{equation*}
\dim ~ \mathbb{F}(X) = \dim ~  \varphi \big( \psi^{-1} \big(
\overline{\Sigma_{d-c+1} ^o (X)} \big) \big) =3.
\end{equation*}
Now, we will show that $X$ is a subset of $\mathbb{F} (X)$. To this
aim, consider the set
\begin{equation*}
\Xi (X) := \{ \mathbb{L}_\mathbb{H} ~|~ \mathbb{H} \in \mathcal{V}
(X) \}.
\end{equation*}
From the the proof of Proposition \ref{prop:dimensionextremal} one
can see that $\Xi (X)$ is $2$-dimensional and $\Sigma_{d-c+1} ^o (X)
\setminus \Xi (X)$ is of dimension at most one. This means that the
coincidence set
\begin{equation*}
Y: = \psi^{-1} \big( \Xi (X) \big) = \{(x,\mathbb{L}_\mathbb{H} )
\mid \mathbb{H} \in \mathcal{V}(X) \mbox{ and } x \in
\mathbb{L}_\mathbb{H} \}
\end{equation*}
is $3$-dimensional and its subset
\begin{equation*}
T := Y \cap \big (X \times \mathbb{G}(1,\mathbb{P}^r)\big) =
\{(x,\mathbb{L}_\mathbb{H} ) \mid \mathbb{H} \in \mathcal{U}(X)
\mbox{ and } x \in X \cap \mathbb{L}_\mathbb{H} \}
\end{equation*}
is $2$-dimensional. Now, the previous disjointness implies that the
projection map
\begin{equation*}
\varphi \upharpoonright: T \rightarrow X \cap \mathbb{F}(X)
\end{equation*}
is injective. It follows that $X \cap \mathbb{F}(X)$ is
$2$-dimensional and so $X \subset \mathbb{F}(X)$.

Next, we will prove that $b_\mathbb{H} = 1$. Recall that
$\mathbb{L}_{\mathbb{H}'}$ and $\mathbb{L}_\mathbb{H}$ are disjoint
if $\mathbb{L}_{\mathbb{H}'} \neq \mathbb{L}_\mathbb{H}$. This
implies that the line $\mathbb{L}_{\mathbb{H}'}$ avoids the vertex
$q$ of the threefold rational normal scroll $W_\mathbb{H} \cap \mathbb{H}' =
S(0,b_\mathbb{H},r-3-b_\mathbb{H})$ and that it is not contained in
any of the ruling planes of $W_\mathbb{H} \cap \mathbb{H}'$.
Therefore $\mathbb{M}_{\mathbb{H}'} := \pi_{\mathbb{L}_\mathbb{H}}
(\mathbb{L}_{\mathbb{H}'})$ is a line such that $g(C_{\mathbb{H}'}
\cap \mathbb{L}_{\mathbb{H}'}) \subset Z_\mathbb{H} \cap
\mathbb{M}_{\mathbb{H}'}$. As $g$ is an isomorphism, we have
\begin{equation*}
\rm{length} (C' \cap \mathbb{M}_{\mathbb{H}'}) =  \rm{length}
(C_{\mathbb{H}'} \cap \mathbb{L}_{\mathbb{H}'}) = d-r+3 > 2.
\end{equation*}
Then $\mathbb{M}_{\mathbb{H}'} \subset Z_\mathbb{H}$ since
$Z_\mathbb{H}$ is cut out by quadrics (cf. Notation and Remarks
\ref{nar:descriptionscrolls}(A)). Furthermore,
$\mathbb{M}_{\mathbb{H}'}$ is not a ruling line of $Z_\mathbb{H}$
since the ruling lines of $Z_\mathbb{H}$ are precisely the images of
the ruling planes of $W_\mathbb{H} \cap \mathbb{H}'$ by the linear
projection map from $q$. Consequently, $\mathbb{M}_{\mathbb{H}'}$
must be a line section of $Z_\mathbb{H}$. In particular,
$b_\mathbb{H} = 1$.

Now, we will show that $r=5$. Assume to the contrary, that $r \geq
6$. Note that $\mathbb{L}_{\mathbb{H}'}$ is contained in
$W_\mathbb{H} \cap \mathbb{H}' = S(0,1,r-4)$ while it is not
contained in any of the ruling planes of $W_\mathbb{H} \cap
\mathbb{H}'$. This means that $\mathbb{L}_{\mathbb{H}'}$ is
contained in the plane $S(0,1)$ of $W_\mathbb{H} \cap \mathbb{H}'$.
It follows that the $3$-space $S(0,0,1)$ of $W_\mathbb{H}$ contains
$\mathbb{L}_{\mathbb{H}'}$ for all $\mathbb{H}' \in \mathcal{V}(X)$
and hence $\mathbb{F}(X) \subset S(0,0,1)$. This is impossible since
$\mathbb{F}(X)$ contains $X$ and hence it spans $\mathbb{P}^r$. This
contradiction shows that indeed $r=5$.

Finally, note that since $b_\mathbb{H} > 0$ we have $b_{\mathbb{H}'}
> 0$ for all $\mathbb{H}' \in \mathcal{V}(X)$ by $(a)$. Then
applying the previous arguments to $\mathbb{H}'$, we get the desired
equality $b_{\mathbb{H}'} =1$.
\end{proof}

Now, we are ready to prove the following complete classification
result of surfaces of maximal sectional regularity.

\begin{theorem}\label{thm:classificationsurface}
Let $5 \leq r < d$ and let $X \subset \mathbb{P}^r$ be a
nondegenerate irreducible projective surface of degree $d \geq r+1$.
If $X$ is a surface of maximal sectional regularity, then either
$\mathbb{F}(X)$ is a plane or else $r=5$, $\mathbb{F}(X) = S(1,1,1)$
and $X \subset \mathbb{F}(X)$. Moreover we have
\begin{itemize}
\item[\rm{(a)}] The following two statements are equivalent:
\begin{itemize}
\item[\rm{(i)}] $X$ is a surface of maximal sectional regularity and $\mathbb{F}(X)$ is a plane.
\item[\rm{(ii)}] $X$ is either a cone over a curve of maximal regularity or else a projection $\pi_{\Lambda} (\widetilde{X})$
of a smooth rational normal surface scroll $\widetilde{X}  \subset
\mathbb{P}^{d+1}$ where
\begin{itemize}
\item[1.] $D \subset \widetilde{X}$ is an effective divisor linearly equivalent to
$H + (1-c)F$ where $H$ and $F$ are respectively a hyperplane divisor
and a ruling line of $\widetilde{X}$ $($hence $\langle D \rangle =
\mathbb{P}^{d-r+3})$, and
\item[2.] $\Lambda \subset \langle D \rangle$ is a $(d-r)$-dimensional subspace such
that the restriction $\pi_{\Lambda} \upharpoonright :\langle D
\rangle \setminus \Lambda \twoheadrightarrow \mathbb{P}^2$ is
generically injective along $D$.
\end{itemize}
\end{itemize}
In this case, $X$ is singular.
\item[\rm{(b)}] The following two statements are equivalent:
\begin{itemize}
\item[\rm{(i)}] $X$ is a surface of maximal sectional regularity and $\mathbb{F}(X) = S(1,1,1)$.
\item[\rm{(ii)}] $X$ is contained in $S(1,1,1)$ as a divisor linearly equivalent to $H + (d-3)F$, where $H$ is the hyperplane divisor and $F$ is a ruling plane of $S(1,1,1)$.
\end{itemize}
In this case, $X$ is smooth.
\end{itemize}
\end{theorem}

\begin{proof}
Suppose that $\mathbb{F}(X)$ is not a plane. Then Proposition
\ref{prop:structureofF(X)} shows that $r=5$, ${\rm
dim}~\mathbb{F}(X) = 3$ and $X \subset \mathbb{F}(X)$. Thus it
remains to prove that $\mathbb{F}(X)= S(1,1,1)$. To this aim, let
$\mathbb{H}, \mathbb{H}' \in \mathcal{V}(X)$ such that
$\mathbb{L}_{\mathbb{H}} \neq \mathbb{L}_\mathbb{H}'$. Then
Proposition \ref{prop:structureofF(X)}(b) says that the two quadrics
$W_\mathbb{H}$ and $W_{\mathbb{H}'}$ in $\mathbb{P}^5$ are both of
type $S(0,0,1,1)$ and $\langle \mathbb{L}_\mathbb{H},
\mathbb{L}_{\mathbb{H}'} \rangle$ is a $3$-space. Moreover, as
$\mathbb{L}_\mathbb{H}$ is the vertex of $W_\mathbb{H} = S(0,0,1,1)$
and $\mathbb{L}_{\mathbb{H}'} \subset W_\mathbb{H}$ it holds
$\langle \mathbb{L}_\mathbb{H}, \mathbb{L}_{\mathbb{H}'} \rangle
\subset W_\mathbb{H}$. Hence, by symmetry we get
\begin{equation*}
X \subset \mathbb{F}(X) \subset W_\mathbb{H} \cap W_{\mathbb{H}'}
\quad \mbox{and} \quad \mathbb{P}^3 = \langle
\mathbb{L}_\mathbb{H},\mathbb{L}_{\mathbb{H}'}\rangle \subset
W_\mathbb{H} \cap W_{\mathbb{H}'}.
\end{equation*}
As $W_\mathbb{H}$ and $W_{\mathbb{H}'}$ are two distinct integral
hyperquadrics in $\mathbb{P}^5$, we know that $W_\mathbb{H} \cap
W_{\mathbb{H}'}$ is a complete intersection and hence a
$3$-dimensional arithmetically Cohen-Macaulay scheme of degree $4$.
As $X \subset \mathbb{P}^5$ is non-degenerate and contained in
$W_\mathbb{H} \cap W_{\mathbb{H}'}$, it follows that
\begin{equation*}
W_\mathbb{H} \cap W_{\mathbb{H}'} = \langle \mathbb{L}_\mathbb{H},
\mathbb{L}_{\mathbb{H}'} \rangle \cup V
\end{equation*}
for a $3$-dimensional non-degenerate integral closed subscheme $V
\subset \mathbb{P}^5$ of degree $3$. Note that $V$ is a scroll of
type $S(1,1,1)$ or $S(0,1,2)$ or $S(0,0,3)$. In particular, $V$ is
cut out by quadrics. Also $X \subset V$. Therefore
$\mathbb{L}_{\mathbb{H}''} \subset V$ for all $\mathbb{H}'' \in
\mathcal{V}(X)$ and hence that $\mathbb{F}(X) = V$. Finally, we aim
to exclude the latter two cases. First, $V$ cannot be equal to
$S(0,0,3)$ since any divisor of $S(0,0,3)$ is arithmetically
Cohen-Macaulay while $X$ is not. Now, assume that $V = S(0,1,2)$ and
let $\mathbb{H}'' \in \mathcal{V}(X)$ be a general member. Then, we
have
\begin{equation*}
C_{\mathbb{H}''} \subset V \cap \mathbb{H}'' = S(1,2) \subset
\mathbb{H}'' = \mathbb{P}^4,
\end{equation*}
and according to Proposition \ref{prop:2.4}.(c), the line section
$S(1)$ of $S(1,2)$ coincides with $\mathbb{L}_{\mathbb{H}''}$.
Therefore $\mathbb{L}_{\mathbb{H}''}$ should lie on the plane
$S(0,1)$ and hence this plane is $\mathbb{F}(X)$, which is a
contradiction. Therefore $\mathbb{F}(X) = V$ is equal to $S(1,1,1)$.
\smallskip

\noindent (a) This follows immediately from Theorem
\ref{thm:linearextremalvariety}.
\smallskip

\noindent (b) (i) $\Longrightarrow$ (ii) : Note that $X$ is
contained in $S(1,1,1)$ since $\mathbb{F}(X)$ is not a plane. For
all $\mathbb{H} \in \mathcal{U}(X)$, we have
\begin{equation*}
C_\mathbb{H} = X \cap \mathbb{H}  \subset \mathbb{F}(X) \cap
\mathbb{H} = S(1,2) \subset \mathbb{P}^4
\end{equation*}
and Proposition \ref{prop:2.4}.(b) yields that the divisor $X$ is
linearly equivalent to $H + (d-3)F$.

(i) $\Longleftarrow$ (ii) : It is clear that any line in the
$2$-dimensional family of line sections of $S(1,1,1)$ is a
$(d-r+3)$-secant line to $X$. Then Theorem
\ref{thm:maximaldimension}.(b) shows that
$\overline{\mathfrak{d}}_{d-r+3}(X)$ is equal to $2$ and so $X$ is a
surface of maximal sectional regularity. Furthermore, $S(1,1,1)$ is
contained in $\mathbb{F}(X)$ and hence we get the desired equality
$S(1,1,1) = \mathbb{F}(X)$ by our previous classification result of
$\mathbb{F}(X)$.

It remains to show that $X$ is smooth. One can easily check that the
$\Delta$-genus of $(X,\mathcal{O}_X (1))$ is equal to zero. This
implies that the linearly normal embedding of $X$ by $\mathcal{O}_X
(1)$, say $\widetilde{X} \subset \mathbb{P}^{d+1}$, is a rational
normal surface scroll and $X$ is the image of an isomorphic linear
projection of $\widetilde{X}$. Since $\widetilde{X}$ admits an
isomorphic linear projection, it is not a cone and hence a smooth
rational normal surface scroll. Therefore $X$ is a smooth surface.
\end{proof}

\begin{remark}\label{rmk:threequadrics}
Let $X$ be as in Theorem \ref{thm:classificationsurface}(b)(ii).
Thus it is contained in $Y:= S(1,1,1)$ as a divisor linearly
equivalent to $H + (d-3)F$, where $H$ is the hyperplane divisor and
$F$ is a ruling plane of $S(1,1,1)$. One can easily check that $H^0
(Y,\mathcal{O}_Y (2H-X))=0$ (cf. Notation and Remarks
\ref{nar:descriptionscrolls}(B)). From the exact sequence
\begin{equation*}
0 \rightarrow \mathcal{I}_Y \rightarrow \mathcal{I}_X \rightarrow
\mathcal{O}_Y (-X) \rightarrow 0
\end{equation*}
it follows that $H^0 (\mathbb{P}^5 ,\mathcal{I}_Y (2))=H^0
(\mathbb{P}^5 ,\mathcal{I}_X (2))$. In particular, $I_X$ requires
exactly three quadratic generators.
\end{remark}

\section{Classification of varieties of maximal sectional regularity}

\noindent This section is aimed to prove the following
classification of varieties of maximal sectional regularity in the
case where the dimension and the codimension are both at least
three.

\begin{theorem}\label{thm:classificationhigherdimensional}
Suppose that $\rm{char}(\Bbbk)=0$ and let $X \subset \mathbb{P}^r$
be a nondegenerate irreducible projective variety of dimension $n
\geq 3$, codimension $c \geq 3$ and degree $d \geq c+3$. If $X$ is a
variety of maximal sectional regularity, then either $\mathbb{F}(X)$
is an $n$-dimensional linear space or else $c=3$ and $\mathbb{F}(X)
= S(\underbrace{0,\ldots,0}_{(n-2)-\rm{times}},1,1,1)$. In the first
case, $X \cap \mathbb{F}(X)$ in $\mathbb{F}(X)$ is a hypersurface of
degree $d-c+1$. In the second case, $\mathbb{F}(X)$ contains $X$.
Moreover we have

\begin{itemize}
\item[\rm{(a)}] The following two statements are equivalent:
\begin{itemize}
\item[\rm{(i)}] $X$ is a variety of maximal sectional regularity and $\mathbb{F}(X)$ is an $n$-dimensional linear space.
\item[\rm{(ii)}] Either $X$ is a cone over a curve of maximal regularity or else $X =\pi_{\Lambda} (\widetilde{X})$ where
    \begin{itemize}
\item[1.] $\widetilde{X} = S(\underbrace{0, \ldots , 0}_{k-\rm{times}} , a_{k+1} , \ldots , a_n) \subset \mathbb{P}^{d+n-1}$ is a rational normal $n$-fold scroll for some $0 \leq k \leq n-2$ and positive integers $a_{k+1} \leq \ldots \leq a_n$,
\item[2.] $D \subset \widetilde{X}$ is an effective divisor linearly equivalent to $H + (1-c)F$ where $H$ is a hyperplane divisor of $\widetilde{X}$ and $F \subset \widetilde{X}$ is an $(n-1)$-dimensional linear space $($hence $\langle D \rangle = \mathbb{P}^{d-r+2n-1})$, and
\item[3.] $\Lambda \subset \langle D \rangle$ is a $(d-c-2)$-dimensional subspace such that the restriction $\pi_{\Lambda} \upharpoonright :\langle D \rangle \setminus \Lambda \twoheadrightarrow \mathbb{P}^n$ is generically injective along $D$.
\end{itemize}
\end{itemize}
\item[\rm{(b)}] The following two statements are equivalent:
\begin{itemize}
\item[\rm{(i)}] $X$ is a variety of maximal sectional regularity and $\mathbb{F}(X) = S(\underbrace{0,\ldots,0}_{(n-2)-\rm{times}},1,1,1)$.
\item[\rm{(ii)}] $X$ is contained in the $(n+1)$-fold scroll $Y:=S(\underbrace{0,\ldots,0}_{(n-2)-\rm{times}},1,1,1)$ as a divisor linearly equivalent to $H + (d-3)F$, where $H$ is the hyperplane divisor of $Y$ and $F \subset Y$ is an $n$-dimensional linear space.
\end{itemize}
\end{itemize}
\end{theorem}

To prove this theorem, we need the following two lemmas.

\begin{lemma}\label{lemma:linearity}
Let $X \subset \mathbb{P}^r$ be a variety of maximal sectional
regularity of dimension $n \geq 2$ and codimension $c \geq 3$. If
$\mathbb{F}(S)$ is a plane for a general linear surface section $S
\subset \mathbb{P}^{r-n+2}$ of $X$, then $\mathbb{F}(X)$ is an
$n$-dimensional linear space.
\end{lemma}

\begin{proof}
We use induction on $n \geq 2$.

The statement is obvious for $n=2$.

Suppose that $n \geq 3$. By induction hypothesis, if $\mathbb{H}_1$
and $\mathbb{H}_2$ are general hyperplanes of $\mathbb{P}^r$ then
$\Lambda_1 := \mathbb{F}(X \cap \mathbb{H}_1)$ and $\Lambda_2 :=
\mathbb{F}(X \cap \mathbb{H}_2)$ are $(n-1)$-dimensional linear
spaces. Furthermore, $\mathbb{F}(X \cap \mathbb{H}_1 \cap
\mathbb{H}_2)$ is an $(n-2)$-dimensional linear subspace (either by
Notation and Remarks \ref{nar:curveofmaxreg}(A) for $n=3$ and by
induction if $n >3$), which is contained in $\Lambda_1 \cap
\Lambda_2$. Thus the linear space
\begin{equation*}
\Lambda := \langle \Lambda_1 , \Lambda_2 \rangle
\end{equation*}
is of dimension $n$. Now, let $\mathbb{H}$ be a general hyperplane.
Then $\mathbb{F}(X \cap \mathbb{H})$ is an $(n-1)$-dimensional
linear space. Also we know that $\mathbb{F}(X \cap \mathbb{H} \cap
\mathbb{H}_1)$ and $\mathbb{F}(X \cap \mathbb{H} \cap \mathbb{H}_2)$
are $(n-2)$-dimensional linear subspaces (again, either by Notation
and Remarks \ref{nar:curveofmaxreg}(A) for $n=3$ or by induction if
$n >3$) and hence that
\begin{equation*}
\mathbb{F}(X \cap \mathbb{H}) = \langle \mathbb{F}(X \cap \mathbb{H}
\cap \mathbb{H}_1) , \mathbb{F}(X \cap \mathbb{H} \cap \mathbb{H}_2)
\rangle.
\end{equation*}
Since $\mathbb{F}(X \cap \mathbb{H} \cap \mathbb{H}_1)$ and
$\mathbb{F}(X \cap \mathbb{H} \cap \mathbb{H}_2)$ are respectively
linear subspaces of $\Lambda_1$ and $\Lambda_2$, it follows that
$\mathbb{F}(X \cap \mathbb{H})$ is contained in $\Lambda$. By Lemma
\ref{lem:planaritycriterion}, we conclude that $\mathbb{F}(X)$ is
equal to the $n$-dimensional linear space $\Lambda$.
\end{proof}

\begin{lemma}\label{lemma:threefoldexceptionalcase}
Suppose that $\rm{char}(\Bbbk)=0$ and let $X \subset \mathbb{P}^6$
be a $3$-dimensional variety of maximal sectional regularity such
that $\mathbb{F}(X \cap \mathbb{H})=S(1,1,1)$ for a hyperplane
$\mathbb{H}=\mathbb{P}^5$. Then
\begin{itemize}
\item[\rm{(a)}] $\mathbb{F}(S)=S(1,1,1)$ for a general hyperplane section $S \subset \mathbb{P}^5$ of $X$.
\item[\rm{(b)}] $\mathbb{F}(X)$ is the rational normal fourfold scroll $S(0,1,1,1)$ and $X$ is contained in $\mathbb{F}(X)$ as a divisor linearly equivalent to $H + (d-3)F$, where $H$ is a hyperplane divisor of $S(0,1,1,1)$ and $F$ is a linear $3$-space in $S(0,1,1,1)$.
\end{itemize}
\end{lemma}

\begin{proof}
\rm{(a)} By Theorem \ref{thm:classificationsurface}(c) and Corollary
\ref{cor:singularityofsurface}, we have
\begin{equation*}
\chi (\mathcal{O}_{X \cap \mathbb{H}} (t)) = d {{t+1} \choose {2}} +
t + 1.
\end{equation*}
Clearly, this implies that the Euler-Poincar$\acute{e}$
characteristic $\chi (\mathcal{O}_X (t))$ is of the form
\begin{equation*}
\chi (\mathcal{O}_X (t)) = d {{t+2} \choose {3}}  + {{t+1} \choose
{2}} + t + \chi_0 (\mathcal{O}_X (1))
\end{equation*}
and hence for a general linear surface section $S \subset
\mathbb{P}^5$ of $X$, we have
\begin{equation*}
\chi (\mathcal{O}_S (t)) = d {{t+1} \choose {2}}  + t + 1.
\end{equation*}
Thus $S$ is a smooth surface by Corollary
\ref{cor:singularityofsurface}. Now, from the classification result
in Theorem \ref{thm:classificationsurface} it follows that
$\mathbb{F}(S)=S(1,1,1)$.

\noindent \rm{(b)} Recall that $X$ is a linear projection of a
threefold rational normal scroll $\widetilde{X} \subset \mathbb{P}^{d+2}$ (cf.
Corollary \ref{cor:linearprojection}). If $\widetilde{X}$ is a cone
over the rational normal curve, then $X$ should be a cone over a
curve of maximal regularity and hence $\mathbb{F}(S)$ is a plane for
a general hyperplane section $S \subset \mathbb{P}^5$ of $X$, a
contradiction. Thus $\rm{Sing}(\widetilde{X})$ is at most a point.
By combining this with the fact that $\dim~\rm{Sing}(\pi_{\Lambda})
\leq 1$ (cf. Theorem \ref{thm:sectionallyrational}), we know that
$\dim ~ \rm{Sing}(X) \leq 1$. Thus for general $\Lambda \in
\mathcal{U}(X)$, the intersection $X \cap \mathbb{L}_{\Lambda}$ is
contained in the smooth locus of $X$ and hence the join $Q_{\Lambda}
:= \rm{Join}(\mathbb{L}_{\Lambda},X)$ is a quadratic hypersurface of
$\mathbb{P}^6$ (cf. Notation and Remark
\ref{notrmk:fourfoldscroll}). Now, choose two general members
$\Lambda_1 , \Lambda_2 \in \mathcal{U}(X)$. Thus we have
\begin{equation*} X \subset Q_{\Lambda_1} \cap Q_{\Lambda_2} \subset
\mathbb{P}^6 .
\end{equation*}
We will see that the intersection $Q_{\Lambda_1} \cap Q_{\Lambda_2}$
is reducible. Indeed, let $\mathbb{H}$ be a general hyperplane and
$Q_{\Lambda_i ,\mathbb{H}}=Q_{\Lambda_i} \cap \mathbb{H}$ ($i=1,2$).
Then
\begin{equation*}
S := X \cap \mathbb{H} \subset Q_{\Lambda_1 ,\mathbb{H}} \cap
Q_{\Lambda_2 ,\mathbb{H}}.
\end{equation*}
Note that $I(S)_2 = I(\mathbb{F}(S))_2$ (cf. Remark
\ref{rmk:threequadrics}) and hence $Q_{\Lambda_1 ,\mathbb{H}}$ and
$Q_{\Lambda_2 ,\mathbb{H}}$ are contained in $I(\mathbb{F}(S))_2$.
Therefore the intersection $Q_{\Lambda_1 ,\mathbb{H}} \cap
Q_{\Lambda_2 ,\mathbb{H}}$ is the union of the scroll
$\mathbb{F}(S)=S(1,1,1)$ and a linear subspace $\mathbb{P}^3$ of
$\mathbb{H}$. Then, it is clear that $Q_{\Lambda_1} \cap
Q_{\Lambda_2}$ should be the union of a scroll $Y:=S(0,1,1,1)$ and a
$4$-dimensional linear space. Obviously, our $X$ is contained in $Y$
as a divisor. Also from the divisor class of $S$ in $S(1,1,1)$, we
know that $X$ is linearly equivalent to $H+(d-3)F$. Finally,
$\mathbb{F}(X) \subset Y$ since $Y$ contains $X$ and it is cut out
by quadrics. On the other hand, $Y \cap \mathbb{H} = \mathbb{F}(X
\cap \mathbb{H}) \subset \mathbb{F}(X)$ for a general hyperplane
$\mathbb{H}$ of $\mathbb{P}^6$. So, we get the desired equality
$\mathbb{F}(X)=Y$.
\end{proof}

Now we give the\\

\noindent {\bf Proof of Theorem
\ref{thm:classificationhigherdimensional}.} Let $S \subset
\mathbb{P}^{c+2}$ be a general linear surface section of $X$. Then
$S$ is a surface of maximal sectional regularity and either
$\mathbb{F}(S)$ is a plane or else $c=3$, $\mathbb{F}(S) = S(1,1,1)$
and $S \subset \mathbb{F}(S)$ (cf. Theorem
\ref{thm:classificationsurface}).

If $\mathbb{F}(S)$ is a plane for a general linear surface section
$S \subset \mathbb{P}^{c+2}$, then $\mathbb{F}(X)$ is an
$n$-dimensional linear space by Lemma \ref{lemma:linearity}.

Now, consider the case where $c=3$ and there exists a
$3$-dimensional linear section $T \subset \mathbb{P}^6$ of $X$ which
is a variety of maximal sectional regularity and which has an
integral hyperplane section $S \subset \mathbb{P}^5$ of maximal
sectional regularity such that $\mathbb{F}(S)=S(1,1,1)$. Then Lemma
\ref{lemma:threefoldexceptionalcase}(b) shows that
$\mathbb{F}(T)=S(0,1,1,1)$ and $T$ is contained in $\mathbb{F}(T)$
as a divisor linearly equivalent to $H + (d-3)F$, where $H$ is a
hyperplane divisor of $S(0,1,1,1)$ and $F$ is a linear $3$-space in
$S(0,1,1,1)$. We first claim that for a general linear subspace
$\mathbb{P}^5 \subset \mathbb{P}^{n+3}$, the surface $X \cap
\mathbb{P}^5$ satisfies the property that $\mathbb{F}(X \cap
\mathbb{P}^5)=S(1,1,1)$. Indeed, by the same argument as in the
proof of Lemma \ref{lemma:threefoldexceptionalcase}(a), $S$ and $X
\cap \mathbb{P}^5$ have the same Euler-Poincar$\acute{e}$
characteristic. Thus our claim is verified by Theorem
\ref{thm:classificationsurface}. Now, observe that $\rm{depth}(T)$
is equal to $2$ (cf. \cite[Theorem 4.3]{P}) and hence
$\rm{depth}(X)= n-1$. In particular, $I_X$ and $I_S$ require the
same number of quadratic generators. Then it follows by Remark
\ref{rmk:threequadrics} that $I_X$ contains exactly three
$\Bbbk$-linearly independent quadrics. Let $\{ Q_1 , Q_2 , Q_3 \}$
be a basis for $H^0 (\mathbb{P}^{n+3},\mathcal{I}_X (2))$ and
consider the closed subset $W \subset \mathbb{P}^{n+3}$ defined as
the intersection of the three hyperquadrics $Q_1$, $Q_2$ and $Q_3$.
We claim that
\begin{equation*}
W =S(\underbrace{0,\ldots,0}_{(n-2)-\rm{times}},1,1,1).
\end{equation*}
Indeed, for a general linear subspace $\mathbb{P}^5 \subset
\mathbb{P}^{n+3}$ consider the quadrics $Q_{i,\mathbb{P}^5} := Q_i
|_{\mathbb{P}^5}$ ($i=1,2,3$). Since $\rm{depth}(X)= n-1$, we know
that $\{ Q_{1,\mathbb{P}^5}, Q_{2,\mathbb{P}^5}, Q_{3,\mathbb{P}^5}
\}$ is a basis for $H^0 (\mathbb{P}^5 , \mathcal{I}_{X \cap
\mathbb{P}^5} (2))$. This implies that $W \cap \mathbb{P}^5$ which
is the intersection of the quadrics $Q_{1,\mathbb{P}^5}$,
$Q_{2,\mathbb{P}^5}$ and $Q_{3,\mathbb{P}^5}$ is precisely equal to
the threefold scroll $\mathbb{F}(X \cap \mathbb{P}^5)$. Therefore
$W$ contains a nondegenerate $(n+1)$-dimensional irreducible variety
$W'$ of degree $3$. Since $W'$ is a variety of minimal degree and
hence cut out by exactly three quadrics, we conclude that $W=W'$.
Also since $W \cap \mathbb{P}^5$ is equal to $S(1,1,1)$, it follows
that $W$ is as above. Finally, note that $\mathbb{F}(X) \subset W$
since $W$ is cut out by quadrics. On the other hand, for a general
linear subspace $\mathbb{P}^5 \subset \mathbb{P}^{n+3}$ we have
\begin{equation*}
W \cap \mathbb{P}^5 = \mathbb{F}(X \cap \mathbb{P}^5 ) \subset
\mathbb{F}(X)
\end{equation*}
which means that $W \subset \mathbb{F}(X)$. Therefore $W =
\mathbb{F}(X)$.
\smallskip

\noindent (a) See Theorem \ref{thm:linearextremalvariety}.
\smallskip

\noindent \rm{(b)} (i) $\Longrightarrow$ (ii) : Note that $X$ is
contained in $\mathbb{F}(X)$ since $\mathbb{F}(X)$ is not a linear
space. For general $\Lambda \in \mathcal{U}(X)$, we have
\begin{equation*}
\mathcal{C}_\Lambda = X \cap \Lambda \subset \mathbb{F}(X) \cap
\Lambda = S(1,2) \subset \mathbb{P}^4
\end{equation*}
and Proposition \ref{prop:2.4}.(b) yields that the divisor $X$ is
linearly equivalent to $H + (d-3)F$.

(i) $\Longleftarrow$ (ii) : Let $\mathbb{P}^5 \in
\mathbb{G}(5,\mathbb{P}^{n+3})$ be a general member. Then we have
\begin{equation*}
S:= X \cap \mathbb{P}^5  \subset Y \cap \mathbb{P}^5 =  S(1,1,1)
\subset \mathbb{P}^5
\end{equation*}
where $S$ is apparently a divisor of $S(1,1,1)$ linearly equivalent
to $H_0 + (d-3)F_0$, where $H_0$ is the hyperplane divisor and $F_0$
is a ruling plane of $S(1,1,1)$. Therefore $S$ is a surface of
maximal sectional regularity and $\mathbb{F}(S)=S(1,1,1)$ (cf.
Theorem \ref{thm:classificationsurface}(b)). It follows that
\begin{equation*}
Y \cap \mathbb{P}^5 = S(1,1,1) \subset \mathbb{F}(X)
\end{equation*}
for general $\mathbb{P}^5 \in \mathbb{G}(5,\mathbb{P}^{n+3})$ and
hence $Y \subset \mathbb{F}(X)$. Then we get $Y = \mathbb{F}(X)$
from our classification result of $\mathbb{F}(X)$ in the present
theorem. \qed \\

\noindent {\bf Acknowledgement.} The first named author thanks to the Korea University Seoul, to
the Mathematisches Forschungsinstitut Oberwolfach, to the Martin-Luther Universit\"{a}t
Halle and to the Deutsche Forschungsgemeinschaft for their hospitality and the financial
support provided during the preparation of this work. The second named author was
supported by the Nation Researcher program 2010-0020413 of NRF and MEST. The third
named author was supported by the NRF-DAAD GEnKO Program (NRF-2011-0021014).

\end{document}